\documentclass[reqno,11pt]{amsart}

\usepackage{amsmath,amsthm,amssymb,amsfonts,color}
\usepackage{graphicx}
\usepackage{verbatim}
\usepackage[colorlinks=true,urlcolor=blue,
citecolor=red,linkcolor=blue,linktocpage,pdfpagelabels,bookmarksnumbered,bookmarksopen]{hyperref}
\usepackage[hyperpageref]{backref}

\newtheorem{theorem}{Theorem}[section]
\newtheorem{corollary}[theorem]{Corollary}
\newtheorem{lemma}[theorem]{Lemma}
\newtheorem{proposition}[theorem]{Proposition}
\theoremstyle{definition}
\newtheorem{remark}[theorem]{Remark}
\newtheorem{definition}[theorem]{Definition}
\newtheorem{example}[theorem]{Example}
\newtheorem*{ack}{Acknowledgements}

\author[Brasco]{Lorenzo Brasco}
\author[Ruffini]{Berardo Ruffini}

\address[L. Brasco]{Aix-Marseille Universit\'e, CNRS
\newline\indent
Centrale Marseille, I2M, UMR 7373, 39 Rue Fr\'ed\'eric Joliot Curie
\newline\indent
13453 Marseille, France}
\email{lorenzo.brasco@univ-amu.fr}

\address[B. Ruffini]{Centre de Math\'ematiques Appliqu\'ees, \'Ecole Polytechnique
\newline\indent
91128 Palaiseau, France}
\email{berardo.ruffini@cmap.polytechnique.fr}
\date{\today}


\numberwithin{equation}{section}

\title[Sobolev embeddings \& Torsion function]{Compact Sobolev embeddings and torsion functions}

\begin{document}

\begin{abstract}
For a general open set, we characterize the compactness of the embedding $W^{1,p}_0\hookrightarrow L^q$ in terms of the summability of its torsion function. In particular, for $1\le q<p$ we obtain that the embedding is continuous if and only if it is compact.
The proofs crucially exploit a {\it torsional Hardy inequality} that we investigate in details.
\end{abstract}

\maketitle

\begin{center}
\begin{minipage}{10cm}
\small
\tableofcontents
\end{minipage}
\end{center}

\section{Introduction}\label{section1}

\subsection{Foreword}

Let $1<p<+\infty$ and let $\Omega\subset\mathbb{R}^N$ be an open set.
We note by $W^{1,p}_0(\Omega)$ the completion of $C^\infty_0(\Omega)$ with respect to the norm
\[
u\mapsto \left(\int_\Omega |\nabla u|^p\,dx\right)^\frac{1}{p}.
\]
The aim of the present paper is to investigate the interplay between the continuity (and compactness) of the embedding $W^{1,p}_0(\Omega)\hookrightarrow L^q(\Omega)$ for $1\le q\le p$ and the integrability properties of the {\it $p-$torsion function of $\Omega$}, $w_\Omega$. The latter is the formal solution of
\begin{equation}
\label{eguazione}
-\Delta_p\, w_\Omega=1,\quad \mbox{ in }\Omega,\quad\qquad w_\Omega=0,\quad \mbox{ on }\partial\Omega,
\end{equation} 
where $\Delta_p\, u=\mathrm{div}(|\nabla u|^{p-2}\,\nabla u)$ is the $p-$Laplacian operator. The reader is referred to Section \ref{sec:2} for the precise definition of $p-$torsion function.
\vskip.2cm\noindent
An important contribution in this direction has been recently given by Bucur and Buttazzo, in a different setting. In particular, in \cite{bubu} the two authors considered the Sobolev space
\[
H^{1,p}_0(\Omega):=\left\{u\, :\, \int_{\Omega} |\nabla u|^p\,dx+\int_{\Omega} |u|^p\,dx<+\infty\ \mbox{ and }\ u\equiv 0 \mbox{ in }\mathbb{R}^N\setminus \Omega\right\},
\]
and they characterized the compactness of the embeddings 
\[
H^{1,p}_0(\Omega)\hookrightarrow L^1(\Omega)\qquad \mbox{ and }\qquad H^{1,p}_0(\Omega)\hookrightarrow L^p(\Omega),
\]
in terms of the summability of the formal solution $u_\Omega$ of 
\begin{equation}
\label{modified_torsion}
-\Delta_p\, u_\Omega+u_\Omega^{p-1}=1,\quad \mbox{ in }\Omega,\quad\qquad u_\Omega=0,\quad \mbox{ on }\partial\Omega.
\end{equation}
Namely, in \cite[Theorems 6.1 \& 6.2]{bubu} they proved that 
\[
H^{1,p}_0(\Omega)\hookrightarrow L^1(\Omega) \mbox{ is continuous }\ \Longleftrightarrow\ u_\Omega\in L^{1}(\Omega)\ \Longleftrightarrow\ H^{1,p}_0(\Omega)\hookrightarrow L^1(\Omega) \mbox{ is compact},
\]
and
\[
H^{1,p}_0(\Omega)\hookrightarrow L^p(\Omega) \mbox{ is compact }\quad \Longleftrightarrow\quad \begin{array}{c}
\mbox{for every $\varepsilon>0$, there exists $R>0$}\\
\mbox{ such that }\quad \|u_\Omega\|_{L^\infty(\Omega\setminus B_R)}<\varepsilon.
\end{array}
\]
We stress that in general $W^{1,p}_0(\Omega)\not=H^{1,p}_0(\Omega)$, unless $\Omega$ supports a Poincar\'e inequality of the type
\[
\int_\Omega |u|^p\,dx\le C\,\int_\Omega |\nabla u|^p,\qquad u\in C^\infty_0(\Omega),
\] 
and $\partial \Omega$ has suitable smoothness properties. 

\subsection{Main results}
In order to present our contribution, for every $1\le q\le p$ we introduce the Poincar\'e constant
\[
\lambda_{p,q}(\Omega):=\inf_{u\in C^\infty_0(\Omega)}\left\{\int_\Omega |\nabla u|^p\,dx\, :\, \int_\Omega|u|^q\,dx=1\right\}.
\]
We remark that the continuity of the embedding $W^{1,p}_0(\Omega)\hookrightarrow L^q(\Omega)$ is equivalent to the condition $\lambda_{p,q}(\Omega)>0$.
We then have the following results. For ease of presentation, we find it useful to distinguish between the case $1\le q<p$ and $q=p$.
\begin{theorem}
\label{thm:main}
Let $1<p<+\infty$ and $1\le q<p$. Let $\Omega\subset\mathbb{R}^N$ be an open set.
Then the following equivalences hold true 
\[
\lambda_{p,q}(\Omega)>0 \quad \Longleftrightarrow \quad w_\Omega\in L^{\frac{p-1}{p-q}\,q}(\Omega)\quad\ \Longleftrightarrow\quad W^{1,p}_0(\Omega)\hookrightarrow L^q(\Omega) \mbox{ is compact}.
\]
Moreover, we have
\begin{equation}
\label{lambda1}
1\le \lambda_{p,q}(\Omega)\,\left(\int_\Omega w_\Omega^{\frac{p-1}{p-q}\,q}\,dx\right)^\frac{p-q}{q}\le \frac{1}{q}\, \left(\frac{p-1}{p-q}\right)^{p-1}.
\end{equation}
\end{theorem}
\noindent
In the case $p=q$, the equivalence 
\[
\lambda_{p,q}(\Omega)>0 \qquad \Longleftrightarrow\qquad W^{1,p}_0(\Omega)\hookrightarrow L^q(\Omega) \mbox{ is compact},
\]
ceases to be true, as shown by simple examples. In this case, by relying on a result by van den Berg and Bucur \cite{vabu}, we obtain the following. 
\begin{theorem}
\label{thm:collection}
Let $1<p<+\infty$ and let $\Omega\subset\mathbb{R}^N$ be an open set. The following equivalence holds true 
\[
\lambda_{p,p}(\Omega)>0\quad \Longleftrightarrow\quad w_\Omega\in L^\infty(\Omega),
\]
and we have the double-sided estimate
\begin{equation}
\label{comprise}
1\le\lambda_{p,p}(\Omega)\,\|w_\Omega\|^{p-1}_{L^\infty(\Omega)}\le \mathcal{D}_{N,p},
\end{equation}
for some constant $\mathcal{D}_{N,p}>1$.
Moreover, we have
\begin{equation}
\label{compatto}
W^{1,p}_0(\Omega)\hookrightarrow L^p(\Omega) \mbox{ is compact }\quad \Longleftrightarrow\quad \begin{array}{c}
\mbox{for every $\varepsilon>0$, there exists $R>0$}\\
\mbox{ such that }\quad \|w_\Omega\|_{L^\infty(\Omega\setminus B_R)}<\varepsilon.
\end{array}
\end{equation}
\end{theorem}
\begin{remark}
For $p=2$ and $q=1$, the result of Theorem \ref{thm:main} is essentially contained in \cite[Theorem 2.2]{bubuve}.
The lower bound in \eqref{comprise} generalizes to $p\not=2$ the estimate of van den Berg in \cite[Theorem 5]{va}. For $\Omega$ smooth and bounded, this was proved in \cite[Proposition 6]{buer} with a different argument.
As for the upper bound, an explicit expression for the costant $\mathcal{D}_{N,p}$ is not given in \cite{vabu}. However, a closer inspection of their proof reveals that it could be possible to produce an explicit value for $\mathcal{D}_{N,p}$ (which is very likely not optimal). In the particular case $p=2$, van den Berg and Carroll in \cite[Theorem 1]{vaca} produced the value
\[
\mathcal{D}_{N,2}=4+3\,\,N\,\log 2.
\]
\end{remark}

\subsection{A comment on the proofs}
Before entering into the mathematical details of the paper, the reader may find it useful to have an idea on some key ingredients of the proofs. In this respect, we wish to mention that a prominent role is played by the {\it torsional Hardy inequality} for general open sets, which is introduced and proved in this paper. The latter is an Hardy-type inequality where the distance function is replaced by the $p-$torsion function. A particular instance is given by
\begin{equation}
\label{THI}
\int_\Omega \frac{|u|^p}{w_\Omega^{p-1}}\,dx\le \int_\Omega |\nabla u|^p\,dx,\qquad \mbox{ for every }u\in W^{1,p}_0(\Omega),
\end{equation}
though we refer the reader to Theorem \ref{thm:delta} and Proposition \ref{prop:HBgeneral} below for a more precise statement and some generalizations. From \eqref{THI} it is then easy to infer for example the lower bounds in \eqref{lambda1} and \eqref{comprise}, when $w_\Omega$ has the required integrability. We also point out that \eqref{THI} holds with constant $1$ and this happens to be optimal.
Observe that inequality \eqref{THI} is dimensionally correct, since equation \eqref{eguazione} entails that $w_\Omega$ scales like a length to the power $p/(p-1)$. 

\subsection{Plan of the paper}
In Section \ref{sec:2} we define the $p-$torsion function of a set and state some general results needed in the sequel. The subsequent Section \ref{sec:3} proves some properties of the $p-$torsion function that will be used throughout the whole paper. The above mentioned torsional Hardy inequality is stated and proved in Section \ref{sec:4}. Then the proofs of Theorems \ref{thm:main} and \ref{thm:collection} are contained in Section \ref{sec:5}. Finally, we conclude the paper by addressing the sharpness issue for the torsional Hardy inequality, which is indeed the content of Section \ref{sec:6}.
For completeness, some known convexity inequalities used in Section \ref{sec:6} are stated in Appendix \ref{sec:A}, mainly without proofs.

\begin{ack}
This work started with a question raised by our {\it Maestro} Giuseppe Buttazzo, about the possibility of having a torsional Hardy inequality. We wish to thank him for this initial input. Dorin Bucur kindly pointed out his paper \cite{vabu}, as well as the paper \cite{buer}, we are grateful to him. We also acknowledge some lively discussions with our friend and colleague Guido De Philippis. 
Finally, both authors have been supported by the {\it Agence Nationale de la Recherche}, through the project ANR-12-BS01-0014-01 {\sc Geometrya}.
\end{ack}

\section{Preliminaries}
\label{sec:2}

\begin{definition}[Torsion function: variational construction]
\label{defi:compact}
Let $1<p<+\infty$ and assume that the embedding $W^{1,p}_0(\Omega)\hookrightarrow L^1(\Omega)$ is compact. Then the following variational problem admits a unique solution 
\begin{equation}
\label{energia}
\min\left\{\frac{1}{p}\,\int_\Omega |\nabla u|^p\,dx-\int_\Omega u\,dx \,:\,u\in W^{1,p}_0(\Omega) \right\}.
\end{equation}
We denote by $w_\Omega$ such a solution.
The function $w_\Omega$ is called {\it $p-$torsion function of $\Omega$}. By optimality, it solves
\[
\left\{
\begin{array}{rcll}
-\Delta_p w_\Omega&=&1,\quad &\text{ in } \Omega.\\
w_\Omega&=&0, &\text{ in } \partial\Omega,
\end{array} 
\right.
\]
where $\Delta_p$ is the $p-$Laplacian operator, i.e. $\Delta_p u:=\mathrm{div}(|\nabla u|^{p-2}\,\nabla u)$.
\end{definition}
The previous boundary value problem is intended in the usual weak sense that is
\begin{equation}
\label{weakeqw}
\int_{\Omega} \left\langle|\nabla w_\Omega|^{p-2}\,\nabla w_\Omega,\nabla\phi\right\rangle\,dx=\int_{\Omega}\phi\,dx,\qquad \mbox{ for any }\phi\in W^{1,p}_0(\Omega).
\end{equation}
The definition of  $w_\Omega$ is linked to an optimal Poincar\'e constant, through the relation
\[
\begin{split}
\left(\frac{p}{p-1}\,\max_{u\in W^{1,p}_0(\Omega)}\right.&\left.\left\{\int_\Omega u\,dx-\frac{1}{p}\,\int_{\Omega}|\nabla u|^p\,dx \right\}\right)^{p-1}\\
&=\max_{u\in W^{1,p}_0(\Omega)\setminus\{0\}} \frac{\left(\displaystyle\int_\Omega |u|\,dx\right)^p}{\displaystyle\int_\Omega |\nabla u|^p\,dx}=:T_p(\Omega).
\end{split}
\]
In analogy with the quadratic case, the quantity $T_p(\Omega)$ is called {\it $p-$torsional rigidity of $\Omega$}. By using the equation \eqref{weakeqw}, one can see that the following relation holds
\[
T_p(\Omega)=\left(\int_\Omega w_\Omega\,dx\right)^{p-1}.
\]
 When the embedding $W^{1,p}_0(\Omega)\hookrightarrow L^1(\Omega)$ fails to be compact, the $p-$torsion function of $\Omega$ is defined as follows (see also \cite{vabu,bubu}).  By $B_R$ we note the open ball centered at the origin and with radius $R>0$.
\begin{definition}[Torsion function: general construction]
\label{defi:general}
Let us define 
\[
R_\Omega:=\inf\{R>0\, :\, |\Omega\cap B_R|>0\}.
\]
Then for every $R> R_\Omega$, we take $w_{\Omega,R}\in W^{1,p}_0(\Omega\cap B_R)$ to be the $p-$torsion function of the bounded open set\footnote{This is well-defined, since in this case $W^{1,p}_0(\Omega\cap B_R)\hookrightarrow L^1(\Omega\cap B_R)$ is compact.} $\Omega\cap B_R$, extended by $0$ outside. By the comparison principle, we get that $w_{\Omega,R}\ge w_{\Omega, R'}$ whenever $R\ge R'$, thus is well posed the definition
\begin{equation}
\label{torsioneee}
w_\Omega(x):=\lim_{R\to\infty} w_{\Omega,R}(x),
\end{equation} 
The limit is intended in the pointwise sense.
\end{definition}
\begin{remark}
Of course, in many situations we could have $|\{x\, :\, w_\Omega(x)=+\infty\}|>0$. This is the case for example of $\Omega=\mathbb{R}^N$, since 
\[
w_{\mathbb{R}^N,R}=w_{B_R}=\frac{\left(R^\frac{p}{p-1}-|x|^\frac{p}{p-1}\right)_+}{A_{N,p}},\qquad \mbox{ where }\ A_{N,p}=\frac{p}{p-1}\, N^\frac{1}{p-1},
\]
and thus in this case $w_{\mathbb{R}^N}$ is the trivial function which is everywhere $+\infty$.
\end{remark}
The first simple result shows that Definition \ref{defi:general} is coherent with the compact case. Indeed in this case \eqref{torsioneee} boils down to the usual torsion function given by Definition \ref{defi:compact}
\begin{lemma}
Let $1<p<+\infty$ and assume that the embedding $W^{1,p}_0(\Omega)\hookrightarrow L^1(\Omega)$ is compact. Then the function defined by \eqref{torsioneee} is the unique solution of \eqref{energia}.
\end{lemma}
\begin{proof}
The first observation is that compactness of the embedding $W^{1,p}_0(\Omega)\hookrightarrow L^1(\Omega)$ entails
\[
T_p(\Omega)<+\infty.
\]
Then we extend each $w_{\Omega,R}$ to $0$ in $\Omega\setminus B_R$, so that $w_{\Omega,R}\in W^{1,p}_0(\Omega)$. By using the equation we obtain
\begin{equation}
\label{uff!}
\int_{\Omega} |\nabla w_{\Omega,R}|^p\,dx=\int_\Omega w_{\Omega,R}\,dx.
\end{equation}
On the other hand, the definition of $p-$torsion implies that
\[
\int_\Omega w_{\Omega,R}\,dx\le \left(\int_\Omega |\nabla w_{\Omega,R}|^p\,dx\right)^\frac{1}{p}\,T_p(\Omega)^\frac{1}{p},
\]
since $w_{\Omega,R}$ is admissible for the variational problem defining $T_p(\Omega)$.
By keeping the two informations together, we get
\[
\int_\Omega |\nabla w_{\Omega,R}|^p\,dx\le T_p(\Omega)^\frac{1}{p-1},\qquad \mbox{ for every }R> R_\Omega.
\] 
This implies that (up to a subsequence) $w_{\Omega,R}$ converges weakly in $W^{1,p}_0(\Omega)$. Since we have also $L^1(\Omega)$ strong convergence (by compactness of the embedding), the limit function has to be the function $w_\Omega$ defined by \eqref{torsioneee}. This shows in particular that $w_\Omega\in W^{1,p}_0(\Omega)\cap L^1(\Omega)$. 
\par
In order to show that $w_\Omega$ coincides with the torsion function, we take $\phi\in C^\infty_0(\Omega)$ and $R_1>R_\Omega$ large enough so that $\mathrm{spt}(\phi)\Subset \Omega\cap B_R$ for every $R\ge R_1$. By minimality of $w_{\Omega,R}$ we get
\[
\frac{1}{p}\,\int_{\Omega} |\nabla w_{\Omega,R}|^p\,dx-\int_\Omega w_{\Omega,R}\,dx\le \frac{1}{p}\,\int_{\Omega} |\nabla \phi|^p\,dx-\int_\Omega \phi\,dx,\qquad \mbox{ for every }R\ge R_1.
\]
By passing to the limit as $R$ goes to $+\infty$ in the left-hand side, we get
\[
\frac{1}{p}\,\int_{\Omega} |\nabla w_{\Omega}|^p\,dx-\int_\Omega w_{\Omega}\,dx\le \frac{1}{p}\,\int_{\Omega} |\nabla \phi|^p\,dx-\int_\Omega \phi\,dx.
\]
For the gradient term, we used the weak lower semicontinuity of the norm in $W^{1,p}_0(\Omega)$.
Finally, by arbitrariness of $\phi\in C^\infty_0(\Omega)$ the previous inequality shows that $w_\Omega$ is the (unique) solution of \eqref{energia}.
\end{proof}
\begin{remark}[Heat-based definition]
In the case $p=2$, the torsion function of an open set $\Omega\subset\mathbb{R}^N$ can also be defined in terms of the heat equation. We briefly recall the construction, by referring for example to \cite{va,vaca} for more details. One considers the initial boundary value problem
\[
\left\{\begin{array}{rcll}
\partial_t u & =& \Delta u,& \mbox{ in } \Omega\times\mathbb R, \\
u&=&0, & \mbox{ on } \partial \Omega\times\mathbb R,\\
u&=&1, & \mbox{ for } t=0.
\end{array}
\right.
\]
If $U_\Omega$ denotes the solution of this problem, we set
\[
W_\Omega(x)=\int_0^\infty U_\Omega(t,x)\,dt,\qquad x\in\Omega.
\]
It is not difficult to see that $W_\Omega$ solves \eqref{weakeqw}. 
For $p\not=2$ such a definition is not available.
\end{remark}
In what follows, for $p\not= N$ we set
\[
p^*=\left\{\begin{array}{cl}
\dfrac{N\,p}{N-p},& \mbox{ if } 1<p<N,\\
+\infty, & \mbox{ if } p>N.
\end{array}
\right.
\]
We will need the following particular family of Gagliardo-Nirenberg inequalities.
\begin{proposition}[Gagliardo-Nirenberg inequalities]
\label{prop:GNS}
Let $1<p<+\infty$ and $1\le q\le p$, then for every $u\in C^\infty_0(\mathbb{R}^N)$ we have:
\vskip.2cm
\begin{itemize}
\item for $p\not= N$ and $q<r\le p^*$
\begin{equation}
\label{sub}
\left(\int_{\mathbb{R}^N} |u|^{r}\,dx \right)^\frac{1}{r}\le C_1\,\left(\int_{\mathbb{R}^N} |u|^q\,dx\right)^\frac{1-\vartheta}{q}\,\left(\int_{\mathbb{R}^N} |\nabla u|^p\,dx\right)^\frac{\vartheta}{p},
\end{equation}
for some $C_1=C_1(N,p,q,r)>0$ and 
\[
\vartheta=\left(1-\frac{q}{r}\right)\,\frac{N\,p}{N\,p+p\,q-N\,q};
\]
\item for $p=N$ and $q<r<\infty$
\begin{equation}
\label{conforme}
\left(\int_{\mathbb{R}^N} |u|^r\,dx\right)^\frac{1}{r}\le C_2\,\left(\int_{\mathbb{R}^N} |\nabla u|^N\,dx\right)^{\frac{r-q}{N\,r}}\, \left(\int_{\mathbb{R}^N} |u|^q\,dx\right)^\frac{1}{r},
\end{equation}
for some $C_2=C_2(N,q,r)>0$.
\end{itemize}
\end{proposition}
\begin{proof}
Inequality \eqref{sub} for $1<p<N$ is well-known and nowadays can be found in many textbooks on Sobolev spaces. The case $p>N$ follows from the well-known {\it Morrey's inequality} (see \cite[Th\'eor\`eme IX.12]{bre})
\[
\|u\|_{L^\infty(\mathbb{R}^N)}\le C\,\left(\|u\|_{L^p(\mathbb{R}^N)}+\|\nabla u\|_{L^p(\mathbb{R}^N)}\right),
\]
combined with a standard homogeneity argument and interpolation in Lebesgue spaces.
\vskip.2cm
On the contrary, the conformal case \eqref{conforme} seems to be more difficult to find in the literature. We provide a simple proof, which is essentially the same as that of the so-called {\it Ladyzhenskaya inequality} (see \cite[Lemma 1, page 10]{La}), corresponding to $q=p=N=2$ and $r=4$.
For every $t>1$ we have
\[
|u|^t\le t\,\int_{-\infty}^{+\infty} |u|^{t-1}\, |u_{x_i}|\,dx_i,\qquad i=1,\dots,N,
\]  
and thus
\[
|u|^\frac{N\,t}{N-1}\le t^\frac{N}{N-1}\,\prod_{i=1}^N \left(\int_{-\infty}^{+\infty} |u|^{t-1}\, |u_{x_i}|\,dx_i\right)^\frac{1}{N-1}.
\]
By integrating over $\mathbb{R}^N$ we get
\[
\begin{split}
\int_{\mathbb{R}^N} |u|^\frac{N\,t}{N-1}\,dx&\le t^\frac{N}{N-1}\,\prod_{i=1}^N \int_{\mathbb{R}^N} \left(\int_{-\infty}^{+\infty} |u|^{t-1}\, |u_{x_i}|\,dx_i\right)^\frac{1}{N-1}\,dx\\
&\le t^\frac{N}{N-1}\,\prod_{i=1}^N \left[\int_{\mathbb{R}^{N-1}} \int_{-\infty}^{+\infty} |u|^{t-1}\,|u_{x_i}|\,dx_i\, d\widehat x_i\right]^\frac{1}{N-1},
\end{split}
\]
where $d\widehat x_i$ denotes integration with respect to all variables but $x_i$. The second inequality is the classical Gagliardo Lemma, see \cite[Lemma 3.3]{Gi}. From the previous estimate, with some elementary manipulations and an application of H\"older inequality we get
\[
\int_{\mathbb{R}^N} |u|^\frac{N\,t}{N-1}\,dx\le t^\frac{N}{N-1}\, \left(\int_{\mathbb{R}^{N}} |u|^{(t-1)\,\frac{N}{N-1}}\,dx\right)\,\left(\int_{\mathbb{R}^N} |\nabla u|^N\,dx\right)^\frac{1}{N-1}.
\]
We now observe that if we take $t> N$, then 
\[
q<(t-1)\,\frac{N}{N-1}<t\,\frac{N}{N-1},
\]
so that by interpolation in Lebesgue spaces
\[
\int_{\mathbb{R}^{N}} |u|^{(t-1)\,\frac{N}{N-1}}\,dx\le \left(\int_{\mathbb{R}^{N}} |u|^q\,dx\right)^{\frac{1-\alpha}{q}\,\frac{(t-1)\,N}{N-1}}\,\left(\int_{\mathbb{R}^{N}} |u|^\frac{N\,t}{N-1}\,dx\right)^{\frac{t-1}{t}\,\alpha},
\]
where
\[
\alpha=\frac{t}{t-1}\,\left(1-\frac{N}{N\,(t-q)+q}\right).
\]
This gives us for $t> N$
\begin{equation}
\label{quasiquasi}
\left(\int_{\mathbb{R}^N} |u|^\frac{N\,t}{N-1}\,dx\right)^{1-\frac{t-1}{t}\,\alpha}\le t^\frac{N}{N-1}\, \left(\int_{\mathbb{R}^{N}} |u|^q\,dx\right)^{\frac{1-\alpha}{q}\,\frac{(t-1)\,N}{N-1}}\,\left(\int_{\mathbb{R}^N} |\nabla u|^N\,dx\right)^\frac{1}{N-1},
\end{equation}
which proves \eqref{conforme} for exponents $r> N^2/(N-1)$. When $q<r\le N^2/(N-1)$, it is sufficient to use once again interpolation in Lebesgue spaces, together with \eqref{quasiquasi}. We leave the details to the reader.
\end{proof}

\section{Properties of the $p-$torsion function}
\label{sec:3}

\subsection{Compact case}

We present some basic properties of the $p-$torsion function when this can be defined variationally, i.e. when the embedding $W^{1,p}_0(\Omega)\hookrightarrow L^1(\Omega)$ is compact.  
For $1<p<N$ we set
\[
\mathcal{S}_{N,p}=\sup_{\phi\in W^{1,p}_0(\mathbb{R}^N)}\left\{\left(\int_{\mathbb{R}^N} |\phi|^{p^*}\,dx\right)^\frac{p}{p^*}\, :\, \int_{\mathbb{R}^N} |\nabla \phi|^p\,dx=1\right\}.
\]
We recall that $\mathcal{S}_{N,p}<+\infty$ and the supremum above is indeed attained.
\begin{proposition}
\label{prop:limitazza}
Let $1<p<+\infty$ and suppose that $W^{1,p}_0(\Omega)\hookrightarrow L^1(\Omega)$ is compact. Then $w_\Omega\in L^\infty(\Omega)$. Moreover, for $1<p<N$ we have
\begin{equation}
\label{apriori}
\|w_\Omega\|_{L^\infty(\Omega)}\le C\, \left(\int_\Omega w_\Omega\,dx\right)^\frac{p'}{N+p'},\qquad \mbox{ with } C=\frac{N+p'}{p'}\,\mathcal{S}_{N,p}^{\frac{N}{N\,(p-1)+p}}.
\end{equation}
\end{proposition}
\begin{proof}
For $p>N$, the result follows directly from \eqref{sub} with $q=1$ and $r=+\infty$.
\par
Let us focus on the case $1<p<N$.
We take $k>0$ and test \eqref{weakeqw} with $\phi=(w_\Omega-k)_+$. This gives
\begin{equation}
\label{identita}
\int_\Omega |\nabla (w_\Omega-k)_+|^p\,dx=\int_\Omega (w_\Omega-k)_+\,dx.
\end{equation}
We introduce the notation $\mu(k):=|\{x\in\Omega\, :\, w_\Omega(x)>k\}|$ and observe that $\mu(k)<+\infty$ for $k>0$, since $w_\Omega\in L^1(\Omega)$.
By combining Sobolev and H\"older inequalities, we get
\[
\int_\Omega |\nabla (w_\Omega-k)_+|^p\,dx\ge \mathcal{S}_{N,p}^{-1}\,\mu(k)^{1-p-\frac{p}{N}}\,\left(\int_\Omega (w_\Omega-k)_+\,dx\right)^p.
\]
Thus from \eqref{identita} we obtain
\[
\left(\int_k^{+\infty} \mu(t)\,dt\right)^{p-1}\le \mathcal{S}_{N,p}\, \mu(k)^{p+\frac{p}{N}-1},
\]
If we set
\[
M(k)=\int_k^\infty \mu(t)\,dt,
\]
the previous estimate can be written as
\[
M(k)^\frac{N}{N+p'}\le \mathcal{C}\, (-M'(k)),\qquad \mbox{ with } \mathcal{C}:=\mathcal{S}_{N,p}^\frac{N}{N\,(p-1)+p},
\]
where $p'=p/(p-1)$. 
If we fix $k_0\ge 0$, this implies that we have
\[
M(k)^\frac{p'}{N+p'}\le M(k_0)^\frac{p'}{N+p'}+\frac{p'}{N+p'}\,\frac{1}{\mathcal{C}}\,(k_0-k),\qquad \mbox{ for every }k\ge k_0.
\]
The previous inequality implies that
\[
M(k)\equiv 0,\qquad \mbox{ for } k\ge \frac{N+p'}{p'}\,\mathcal{C}\,M(k_0)^\frac{p'}{N+p'}+k_0,
\]
and thus
\[
\mu(k)\equiv 0,\qquad \mbox{ for } k\ge \frac{N+p'}{p'}\,\mathcal{C}\,M(k_0)^\frac{p'}{N+p'}+k_0.
\]
This finally gives
\[
0\le w_\Omega(x)\le \frac{N+p'}{p'}\,\mathcal{C}\,M(k_0)^\frac{p'}{N+p'}+k_0=k_0+\frac{N+p'}{p'}\,\mathcal{C}\,\left( \int_\Omega (w_\Omega-k_0)_+\,dx\right)^\frac{p'}{N+p'}.
\]
By arbitrariness of $k_0$ we thus get the $L^\infty-L^1$ estimate \eqref{apriori}, as desired.
\vskip.2cm\noindent
Finally, for the case $p=N$, we start again by testing the equation with $(w_\Omega-k)_+$. Then to estimate the right-hand side of \eqref{identita}, we now use inequality \eqref{conforme} with $q=1$ and $r=2\,N$. This gives
\[
\begin{split}
\left(\int_\Omega (w_\Omega-k)_+^{2\,N}\,dx\right)^\frac{1}{2\,N}&\le C_2\, \left(\int_\Omega |\nabla (w_\Omega-k)_+|^N\,dx\right)^\frac{2\,N-1}{2\,N^2}\,\left(\int_{\Omega} (w_\Omega-k)_+\,dx\right)^\frac{1}{2\,N}\\
&=C_2\, \left(\int_\Omega (w_\Omega-k)_+\,dx\right)^\frac{3\,N-1}{2\,N^2},
\end{split}
\]
thanks to \eqref{identita}, too. Similarly as before, after some manipulations we get
\[
\left(\int_\Omega (w_\Omega-k)_+\,dx\right)^{\frac{N-1}{N}}\le C_2\, \mu(k).
\]
Then the proof goes as in the previous case.
\end{proof}
We list some composition properties of $w_\Omega$ that will be used many times.
\begin{lemma}
\label{lm:VF}
Let $1<p<+\infty$ and suppose that $W^{1,p}_0(\Omega)\hookrightarrow L^1(\Omega)$ is compact. Then: 
\begin{itemize}
\item[{\it (i)}] for every $0<\beta\le (p-1)/p$, we have $w_\Omega^\beta\not\in W^{1,p}(\Omega)$;
\vskip.2cm
\item[{\it (ii)}] $\log( w_\Omega)\not\in W^{1,p}(\Omega)$;
\vskip.2cm
\item[{\it (iii)}] for every $(p-1)/p<\beta<1$, we have $w_\Omega^\beta\in W^{1,p}_0(\Omega)$, provided $w_\Omega^{\beta\,p-p+1}\in L^1(\Omega)$;
\vskip.2cm
\item[{\it (iv)}] for every $\beta\ge 1$, we have $w^\beta_\Omega\in W^{1,p}_0(\Omega)$.
\end{itemize}
\end{lemma}
\begin{proof}
We treat each case separately.
\vskip.2cm\noindent
\underline{\it (i) Case $0<\beta\le (p-1)/p$.} Let us first assume that
\[ 
0<\beta<\frac{p-1}{p}.
\]
In this case, let us define the function
\begin{equation}
\label{phieps}
\varphi_\varepsilon=(w_\Omega+\varepsilon)^{\beta\,p-p+1}-\varepsilon^{\beta\,p-p+1},
\end{equation}
for $\varepsilon>0$. Notice that $\varphi_\varepsilon\in W^{1,p}_0(\Omega)$, since this is the composition of $w_\Omega$ with the $C^1$ function 
\[
\psi_\varepsilon(t)=(t+\varepsilon)^{\beta\,p-p+1}-\varepsilon^{\beta\,p-p+1},\qquad t\ge 0,
\]
which is globally Lipschitz continuous on $[0,+\infty)$ and such that $\psi_\varepsilon(0)=0$.
Plugging $\varphi_\varepsilon$  as a test function in \eqref{weakeqw} we get
\[
(\beta\,p-p+1)\, \int_\Omega |\nabla w_\Omega|^p\, (w_\Omega+\varepsilon)^{\beta\,p-p}\,dx= \int_\Omega \frac{\varepsilon^{p-1-\beta\,p}-(w_\Omega+\varepsilon)^{p-1-\beta\,p}}{\varepsilon^{p-1-\beta\,p}\,(w_\Omega+\varepsilon)^{p-1-\beta\,p}}\,dx,
\]
that is
\begin{equation}
\label{uno}
(p-1-\beta\,p)\,\int_\Omega \frac{|\nabla w_\Omega|^p}{(w_\Omega+\varepsilon)^{p-\beta p}}\,dx=\frac{1}{\varepsilon^{p-1-\beta\,p}}\,\int_\Omega \frac{(w_\Omega+\varepsilon)^{p-1-\beta\,p}-\varepsilon^{p-1-\beta\,p}}{(w_\Omega+\varepsilon)^{p-1-\beta\,p}}\,dx.
\end{equation}
From Proposition \ref{prop:limitazza}, we already know that $w_\Omega\in L^\infty(\Omega)$, then we take 
\[
\tau=\frac{\|w_\Omega\|_{L^\infty(\Omega)}}{2},\qquad \mbox{ so that}\quad A:=|\{x\in \Omega\, :\, w_\Omega>\tau\}|>0.
\]
Then from \eqref{uno} we get
\[
\begin{split}
(p-1-\beta\,p)\,\int_\Omega \frac{|\nabla w_\Omega|^p}{(w_\Omega+\varepsilon)^{p-\beta p}}\,dx&=\frac{1}{\varepsilon^{p-1-\beta\,p}}\,\int_\Omega
\left[1-\left(\frac{\varepsilon}{w_\Omega+\varepsilon}\right)^{p-1-\beta\,p}\right]\,dx\\
&\ge \frac{1}{\varepsilon^{p-1-\beta\,p}}\,\int_{\{w_\Omega>\tau\}}
\left[1-\left(\frac{\varepsilon}{\tau+\varepsilon}\right)^{p-1-\beta\,p}\right]\,dx \\
&=\left[1-\left(\frac{\varepsilon}{\tau+\varepsilon}\right)^{p-1-\beta\,p}\right]\,\frac{A}{\varepsilon^{p-1-\beta\,p}}.\\
\end{split}
\]
By taking the limit as $\varepsilon$ goes to $0$ in the previous estimate and using the Monotone Convergence Theorem, we get
\[
\int_\Omega \frac{|\nabla w_\Omega|^p}{w_\Omega^{p-\beta p}}\,dx=+\infty.
\]
This finally shows that $\nabla w_\Omega^\beta\not\in L^p(\Omega)$ for $0<\beta<(p-1)/p$.
\par
To treat the borderline case $\beta=(p-1)/p$, we insert in \eqref{weakeqw} the test function
\[
\phi=\log (w_\Omega+\varepsilon)-\log(\varepsilon),
\]
for $\varepsilon>0$.
In this case we obtain
\[
\int_\Omega \frac{|\nabla w_\Omega|^p}{(w_\Omega+\varepsilon)}\, dx=\int_\Omega \log\left(1+\frac{w_\Omega}{\varepsilon}\right)\,dx,
\]
and reasoning as before we get again the desired conclusion.
\vskip.2cm\noindent
\underline{\it (ii) The logarithm.}
To prove that $\log( w_\Omega)\not\in W^{1,p}(\Omega)$, it is sufficient to reproduce the proof above with $\beta=0$. 
\vskip.2cm\noindent
\underline{\it (iii) Case $(p-1)/p<\beta<1$.} 
We test once again \eqref{weakeqw} with $\varphi_\varepsilon$ defined in \eqref{phieps}. In this case we get the equality
\[
(\beta\, p-p+1)\int_\Omega |\nabla w_\Omega|^p\,(w_\Omega+\varepsilon)^{\beta\, p-p}\,dx=\int_\Omega \left(\left(w_\Omega+\varepsilon\right)^{\beta p-p+1}-\varepsilon^{\beta p-p+1}\right)\,dx.
\]
From the previous, with simple manipulations and using the subadditivity of $\tau\mapsto \tau^{\beta\,p-p+1}$ we get for every $\varepsilon>0$
\[
\begin{split}
\int_\Omega \left|\nabla \left((w_\Omega+\varepsilon)^\beta-\varepsilon^\beta\right)\right|^p\,dx&\le \frac{\beta^p}{\beta\, p-p+1}\,\int_\Omega \left(\left(w_\Omega+\varepsilon\right)^{\beta\,p-p+1}-\varepsilon^{\beta\,p-p+1}\right)\,dx\\
&\le \frac{\beta^p}{\beta\, p-p+1}\,\int_\Omega w_\Omega^{\beta\,p-p+1}\,dx,
\end{split}
\]
and the latter is finite by hypothesis. Thus the net 
\[
\left\{(w_\Omega+\varepsilon)^\beta-\varepsilon^\beta\right\}_{\varepsilon>0},
\]
is uniformly bounded in $W^{1,p}_0(\Omega)$. Since the latter is a weakly closed space, we get that $w_\Omega^\beta\in W^{1,p}_0(\Omega)$ as desired.
\vskip.2cm\noindent
\underline{\it (iv) Case $\beta \ge 1$.}
This is the simplest case. By Proposition \ref{prop:limitazza} $w_\Omega\in L^\infty(\Omega)$, then $w_\Omega^\beta$ is just the composition of a $C^1$ function vanishing at $0$ with a function in $W^{1,p}_0(\Omega)\cap L^\infty(\Omega)$. This gives $w^\beta_\Omega\in W^{1,p}_0(\Omega)$.
\end{proof}
\begin{remark}
The requirement $w_\Omega^{\beta\,p-p+1}\in L^1(\Omega)$ in point {\it (iii)} of the previous Lemma is necessary. 
Indeed, for every $(p-1)/p<\beta<1$, it is possible to construct an open set $\Omega\subset\mathbb{R}^N$ such that 
$W^{1,p}_0(\Omega)\hookrightarrow L^1(\Omega)$ is compact,
but $w_\Omega^{\beta}\not\in L^1(\Omega)$. An instance of such a set is presented in Remark \ref{rem:trenoedisperazione} below.	 
\end{remark}

\subsection{General case}

We already said that in general $w_\Omega$ could reduce to the trivial function which is $+\infty$ everywhere on $\Omega$. The following elegant and simple result, suggested to us by Guido De Philippis (see \cite{DeltaPhi}), gives a sufficient condition to avoid this trivial situation. 
\begin{lemma}[Propagation of finiteness]
\label{lm:guidone}
Let $1<p<+\infty$ and let $\Omega\subset\mathbb{R}^N$ be an open set. Let us suppose that there exist $R_0>R_\Omega$, $x_0\in\Omega\cap B_{R_0}$ and $M>0$ such that
\begin{equation}
\label{punto}
w_{\Omega,R}(x_0)\le M,\qquad \mbox{ for every } R\ge R_0. 
\end{equation}
Then $w_\Omega\in L^\infty_{loc}(\Omega_{x_0})$, where $\Omega_{x_0}$ is the connected component of $\Omega$ containing $x_0$.
\end{lemma}
\begin{proof}
We first observe that the pointwise condition \eqref{punto} does make sense, since each function $w_{\Omega,R}$ is indeed $C^{1,\alpha}_{loc}(\Omega\cap B_R)$ for some $0<\alpha<1$, thanks to standard regularity results for the $p-$Laplacian. In this respect, a classical reference is \cite{DiB}.
\par
Let $K\Subset \Omega_{x_0}$ be a compact set, then there exists a larger compact set $K\subset K'\Subset\Omega_{x_0}$ such that $x_0\in K'$. We take $R_1\ge R_0$ large enough, so that $\Omega\cap B_{R_1}$ contains $K'$. By Harnack inequality, we have
\[
\begin{split}
\sup_{K} w_{\Omega,R}\le\sup_{K'} w_{\Omega,R}&\le C_{K'}\,\left(\inf_{K'} w_{\Omega,R}+|K'|^{\frac{1}{N}\,\frac{p}{p-1}}\right)\\
&\le C_{K'}\, w_{\Omega,R}(x_0)+C_{K'}\,|K'|^{\frac{1}{N}\,\frac{p}{p-1}}\le C,\qquad \mbox{ for }R\ge R_1,
\end{split}
\]
where $C=C(N,p,K',M)>0$.
This ends the proof.
\end{proof}
In general it is not true that \eqref{punto} implies $w_\Omega\in L^\infty_{loc}(\Omega)$, unless $\Omega$ is connected as shown in the next simple counterexample.
\begin{example}
Let us consider
\[
\Omega=B_1\cup \{x\in\mathbb{R}^N\, :\, x_N>2\}.
\]
In this case we have 
\[
w_\Omega(0)=\frac{p-1}{p}\, N^{-\frac{1}{p-1}}\qquad \mbox{ and }\qquad w_\Omega=+\infty\quad \mbox{ on } \{x\in\mathbb{R}^N\, :\, x_N>2\}.
\]
\end{example}
We present now a sufficient condition for the function $w_\Omega$ defined by \eqref{torsioneee} to be a (local) weak solution of
\[
-\Delta_p w=1.
\]
\begin{proposition}
\label{prop:vera}
Let $1<p<+\infty$ and let $\Omega\subset\mathbb{R}^N$ be an open set. Let us suppose that $w_\Omega\in L^1_{loc}(\Omega)$.
Then 
\begin{equation}
\label{Lp}
\nabla w_\Omega\in L^p_{loc}(\Omega).
\end{equation}
Moreover, $w_\Omega$ is a local weak solution of \eqref{weakeqw}, i.e. for every $\Omega'\Subset\Omega$ and every $\phi\in C^\infty_0(\Omega')$ there holds
\[
\int \left\langle |\nabla w_\Omega|^{p-2}\,\nabla w_\Omega,\nabla\phi\right\rangle\,dx=\int\phi\,dx. 
\]
\end{proposition}
\begin{proof}
To prove \eqref{Lp} it suffices to show that for every open set $\Omega'\Subset\Omega$, there exists a constant $C_{\Omega'}>0$ such that 
\begin{equation}
\label{gradienti}
\|\nabla w_{\Omega,R}\|_{L^{p}(\Omega')}\le C_{\Omega'},\qquad \mbox{ for every } R>\rho_{\Omega'}:=\min\{\rho\in [0,\infty)\, :\, \Omega'\Subset B_\rho\}.
\end{equation}
Indeed, if this were true, by weak convergence (up to a subsequence) of the gradients for every $\phi\in C^\infty_0(\Omega')$  we would get
\[
\lim_{R\to\infty}\int_\Omega\left(w_{\Omega,R}\right)_{x_j}\phi\,dx=-\lim_{R\to\infty}\int_\Omega w_{\Omega,R}\, \phi_{x_j}\,dx=-\int_\Omega w_\Omega\, \phi_{x_j}\,dx,
\]
which implies that $\nabla w_\Omega\in L^p_{loc}(\Omega)$. Observe that we used that $w_\Omega\in L^1_{loc}(\Omega)$ to pass to the limit.
\par
To show the uniform bound \eqref{gradienti}, 
we choose $\Omega'\Subset\Omega_1\Subset\Omega$ and a positive cut-off function $\eta\in C^\infty_0(\Omega_1)$ such that
\[
\eta\equiv 1 \mbox{ on } \Omega',\qquad\qquad|\nabla\eta|\le \frac{4}{\mathrm{dist}(\Omega',\partial \Omega_1)}.
\]  
Then, for a fixed $R\ge \rho_{\Omega_1}$, we insert the test function $\phi=w_{\Omega,R}\,\eta^p$ in the weak formulation of the equation solved by $w_{\Omega,R}$. Observe that this is an admissible test function, since it is supported in $\Omega_1\Subset\Omega\cap B_R$. With simple manipulations, we get
\begin{equation}
\label{caccioppoli}
\begin{split}
\int_{\Omega} |\nabla w_{\Omega,R}|^p\,\eta^p\,dx&\le \int_{\Omega}\eta^p w_{\Omega,R}\,dx+p\, \int_\Omega \eta^{p-1}|\nabla w_{\Omega,R}|^{p-1}\,|\nabla \eta|\,  w_{\Omega,R}\,dx\\
&\le \int_{\Omega_1} w_{\Omega}\,dx+\varepsilon^{1-p}\,\int_{\Omega} |w_{\Omega}|^p\,|\nabla \eta|^p\,dx\\
&+(p-1)\,\varepsilon\,\int_\Omega |\nabla w_{\Omega,R}|^p\,\eta^p\,dx,
\end{split}
\end{equation}
where we also used that $w_{\Omega,R}\le w_\Omega$ by construction.
The last term can be absorbed in the left-hand side of \eqref{caccioppoli} by taking $\varepsilon>0$ small enough.  
Thus we end up with
\begin{equation}
\label{LpLp}
\int_{\Omega'} |\nabla w_{\Omega,R}|^p\,dx\le C\,\int_{\Omega_1} w_{\Omega}\,dx+\frac{C}{\mathrm{dist}(\Omega',\partial \Omega_1)^p}\, \int_{\Omega_1} |w_{\Omega}|^p\,dx,
\end{equation} 
for some $C=C(N,p)>0$,
where we also used the bound on $|\nabla \eta|$. In order to conclude, we need to bound the right-hand side of \eqref{LpLp}. Since we are assuming $w_\Omega\in L^1_{loc}(\Omega)$, then we can apply Lemma \ref{lm:guidone} in each connected component of $\Omega$ and obtain $w_\Omega\in L^\infty_{loc}(\Omega)$. Thus the right-hand side of \eqref{LpLp} is finite and we get \eqref{gradienti}.
\vskip.2cm\noindent
In order to show that $w_\Omega$ is a local weak solution of \eqref{weakeqw}, we need to pass to the limit in the equation
\begin{equation}
\label{equar}
\int \langle |\nabla w_{\Omega,R}|^{p-2}\,\nabla w_{\Omega,R},\nabla\phi\rangle\,dx=\int \phi\,dx.
\end{equation}
We first observe that for $p=2$ the local weak convergence of the gradients already gives the result, by linearity of \eqref{equar}. In the case $p\not=2$ we need to improve this weak convergence into a stronger one. For this, we can use the higher differentiability of solutions of the $p-$Laplacian. Namely, it is sufficient to observe that for every (smooth) open sets $\Omega'\Subset \Omega_1\Subset \Omega$, we have
\begin{equation}
\label{p<2}
\|\nabla w_{\Omega,R}\|_{W^{1,p}(\Omega')}\le \frac{C}{\mathrm{dist}(\Omega',\partial \Omega_1)}\,\|\nabla w_{\Omega,R}\|_{L^p(\Omega_1)},\qquad \mbox{ for } 1<p< 2,
\end{equation}
and
\begin{equation}
\label{p>2}
\left\||\nabla w_{\Omega,R}|^\frac{p-2}{2}\,\nabla w_{\Omega,R}\right\|_{W^{1,2}(\Omega')}^2\le \frac{C}{\mathrm{dist}(\Omega',\partial \Omega_1)^2}\,\|\nabla w_{\Omega,R}\|_{L^p(\Omega_1)}^p,\qquad \mbox{ for } p>2.
\end{equation}
again for $R>\rho_\Omega'$, so that $\Omega'\Subset \Omega\cap B_R$. These estimates are nowadays well-known: the first one comes from \cite[Proposition 2.4]{acfu}, while the second one can be found for example in \cite[Theorem 4.2]{bracarsan}. Observe that the right-hand sides of \eqref{p<2} and \eqref{p>2} are uniformly bounded, thanks to the first part of the proof. 
\par
From \eqref{p<2}, by Sobolev embedding $W^{1,p}(\Omega')\hookrightarrow L^p(\Omega')$ we have strong convergence (up to a subsequence) in $L^{p}(\Omega')$ of $\nabla w_{\Omega,R}$ to $\nabla w_{\Omega}$. If one then uses the elementary inequality\footnote{This follows from the fact that $z\mapsto |z|^{p-2}\,z$ is $(p-1)-$H\"older continuous, for $1<p<2$.}
\[
\int_{\Omega'} \left||\nabla w_{\Omega,R}|^{p-2}\,\nabla w_{\Omega,R}-|\nabla w_{\Omega}|^{p-2}\,\nabla w_{\Omega}\right|^{p'}\,dx\le C\,\int_{\Omega'} \left|\nabla w_{\Omega,R}-\nabla w_{\Omega}\right|^{p}\,dx,
\]
we obtain strong convergence in $L^{p'}(\Omega')$ of $|\nabla w_{\Omega,R}|^{p-2}\,\nabla w_{\Omega,R}$ to $|\nabla w_{\Omega}|^{p-2}\,\nabla w_{\Omega}$. Thus it is possible to pass to the limit in \eqref{equar} for $1<p<2$.
\par
For $p>2$, we observe that \eqref{p>2}, Sobolev embedding $W^{1,2}(\Omega')\hookrightarrow L^2(\Omega')$ and the elementary inequality\footnote{Observe that $z\mapsto |z|^\frac{2-p}{p}\,z$ is $2/p-$H\"older continuous, i.e.
\[
\left||z|^\frac{2-p}{p}\,z-|\xi|^\frac{2-p}{p}\,\xi\right|\le C\, |z-\xi|^\frac{2}{p}.
\]
The desired inequality is obtained by choosing
\[
z=|\nabla w_{\Omega,R}|^\frac{p-2}{2}\,w_{\Omega,R}\qquad \mbox{ and }\qquad \xi=|\nabla w_{\Omega,R'}|^\frac{p-2}{2}\,w_{\Omega,R'}.
\]
}
\[
\int_{\Omega'} \left||\nabla w_{\Omega,R}|^\frac{p-2}{2}\,\nabla w_{\Omega,R}-|\nabla w_{\Omega,R'}|^\frac{p-2}{2}\,\nabla w_{\Omega,R'}\right|^{2}\,dx\ge C\,\int_{\Omega'} \left|\nabla w_{\Omega,R}-\nabla w_{\Omega,R'}\right|^{p}\,dx,
\]
imply again that we can extract a sequence such that the gradients strongly convergence in $L^p(\Omega')$. The limit is of course $w_{\Omega,R}$, since this has to coincide with the weak limit. In order to conclude, we can observe that for every $\phi\in C^\infty(\Omega')$ we have
\[
\begin{split}
\Big|\int \langle |\nabla w_{\Omega,R}|^{p-2}\,w_{\Omega,R}&-|\nabla w_{\Omega}|^{p-2}\,w_{\Omega},\nabla\phi\rangle\,dx\Big|\\
&\le \|\nabla \phi\|_{L^\infty}\, \int_{\Omega'} \left||\nabla w_{\Omega,R}|^{p-2}\,\nabla w_{\Omega,R}-|\nabla w_{\Omega}|^{p-2}\,\nabla w_{\Omega}\right|\,dx\\
&\le C\,\|\nabla \phi\|_{L^\infty}\, \int_{\Omega'} (|\nabla w_{\Omega,R}|^{p-2}+|\nabla w_{\Omega}|^{p-2})\,|\nabla w_{\Omega,R}-\nabla w_\Omega|\,dx. 
\end{split}
\]
From the strong convergence of the gradients in $L^p_{loc}$, we get the desired result.
\end{proof}
\begin{remark}
Though we will not need this, we notice that once we obtained that $w_\Omega\in W^{1,p}_{loc}(\Omega)$ is a local weak solution of the equation, then we have $w_\Omega \in C^{1,\alpha}_{loc}(\Omega)$ for some $0<\alpha<1$.
\end{remark}

\section{The torsional Hardy inequality}
\label{sec:4}

In this section we are going to prove a Hardy-type inequality, which contains weights depending on $w_\Omega$. The proof of the sharpness of such an inequality is postponed to Section \ref{sec:6}.

\subsection{Compact case}
We start with the following slightly weaker result.
\begin{proposition}
\label{thm:disuguaglianzadebole}
Let $1<p<+\infty$ and let $\Omega\subset\mathbb{R}^N$ be an open set such that the embedding $W^{1,p}_0(\Omega)\hookrightarrow L^1(\Omega)$ is compact. Then for every $u\in W^{1,p}_0(\Omega)$ we have 
\begin{equation}
\label{suboptimalHB}
\left(\frac{p-1}{p}\right)^{p}\,\int_\Omega \left[\left|\frac{\nabla w_\Omega}{w_\Omega}\right|^p+\frac{p}{(p-1)}\,\frac{1}{w_\Omega^{p-1}} \right]\,|u|^p\,dx\le \int_\Omega |\nabla u|^p\,dx.
\end{equation}
\end{proposition}
\begin{proof}
We first observe that it is sufficient to prove inequality \eqref{suboptimalHB} for positive functions. Let $u\in C^\infty_0(\Omega)$ be positive. 
We recall that
\begin{equation}
\label{conf}
\int_{\Omega}
\left\langle |\nabla w_\Omega|^{p-2}\,\nabla w_\Omega,\nabla \phi\right\rangle\,dx=\int_{\Omega} \phi\,dx,
\end{equation}
for any $\phi\in W^{1,p}_0(\Omega)$. By taking in \eqref{conf} the test
function 
\[
\phi=u^p\,(w_\Omega+\varepsilon)^{1-p},
\]
we get
\begin{equation}
\label{1-7}
\int_\Omega \left[\frac{(p-1)\,|\nabla w_\Omega|^p+(w_\Omega+\varepsilon)}{(w_\Omega+\varepsilon)^p} \right]\,u^p\,dx=p\,\int_\Omega u^{p-1}\,\left \langle 
\frac{|\nabla w_\Omega|^{p-2}\,\nabla w_\Omega}{(w_\Omega+\varepsilon)^{p-1}},\nabla u \right\rangle\,dx. 
\end{equation}
By Young inequality, for any $\xi,z\in\mathbb{R}^N$ 
it holds
\begin{equation}
\label{giovane!}
\langle \xi,z\rangle \le \frac{1}{p}\, |z|^p+\frac{p-1}{p}\,|\xi|^\frac{p}{p-1}.
\end{equation}
By applying such an inequality to \eqref{1-7}, with
\[
z=\left(\frac{p}{p-1}\right)^\frac{p-1}{p}\,\nabla u,\qquad \mbox{ and }\qquad \xi=\left(\frac{p-1}{p}\right)^{\frac{p-1}{p}}\,u^{p-1}\,\frac{|\nabla w|^{p-2}\, \nabla w}{(w+\varepsilon)^{p-1}},
\]
we get that
\[
\begin{split}
\int_\Omega \left[\frac{(p-1)\,|\nabla w|^p+(w+\varepsilon)}{(w+\varepsilon)^p} \right]\,u^p\,dx&\le \left(\frac{p}{p-1}\right)^{p-1}\,\int_\Omega |\nabla
u|^p\,dx\\
&+(p-1)\,\frac{p-1}{p}\,\int_\Omega \frac{u^p\,|\nabla w|^p}{(w+\varepsilon)^{p}}\,dx.
\end{split}
\]
The previous inequality gives
\[
\begin{split}
\left(\frac{p-1}{p}\right)^p\,\int_\Omega \left[\frac{|\nabla w|^p}{(w+\varepsilon)^p} +\frac{p}{(p-1)\,(w+\varepsilon)^{p-1}}\right]\,u^p\,dx\le \int_\Omega |\nabla
u|^p\,dx.
\end{split}
\]
Finally we let $\varepsilon$ go to $0$, then Fatou's Lemma gives the inequality \eqref{suboptimalHB} for $u\in C^\infty_0(\Omega)$ positive. The case of a general $u\in W^{1,p}_0(\Omega)$ follows by density.  
\end{proof}
As a consequence of the torsional Hardy inequality, we record the following integrability properties of functions in $W^{1,p}_0(\Omega)$. This will be useful in a while.
\begin{corollary}
Under the assumptions of Proposition \ref{thm:disuguaglianzadebole},
for every $u\in W^{1,p}_0(\Omega)$ we have
\begin{equation}
\label{integrability}
\int_\Omega \left|\frac{\nabla w_\Omega}{w_\Omega}\right|^p\,|u|^p\,dx<+\infty\qquad \mbox{ and }\qquad \int_\Omega \frac{|u|^p}{w_\Omega^{p-1}}\,dx<+\infty.
\end{equation}
\end{corollary}
The following functional inequality is the main result of this section.
\begin{theorem}[Torsional Hardy inequality]
\label{thm:delta}
Under the assumptions of Proposition \ref{thm:disuguaglianzadebole}, for every $\delta>0$ and every $u\in W^{1,p}_0(\Omega)$ we have
\begin{equation}
\label{HardyButtazzo}
\frac{p-1}{\delta}\,\int_\Omega \left[\left(1-\delta^{-\frac{1}{p-1}}\right)\,\left|\frac{\nabla w_\Omega}{w_\Omega}\right|^p+\frac{1}{(p-1)\,w_\Omega^{p-1}} \right]\,|u|^p\,dx\le \int_\Omega |\nabla
u|^p\,dx.
\end{equation}
\end{theorem}
\begin{proof}
The proof is the same as that of Proposition \ref{thm:disuguaglianzadebole}. The main difference is that now we use Young inequality \eqref{giovane!} with the choices
\begin{equation}
\label{ab}
z=\delta^\frac{1}{p}\,\nabla u\qquad\mbox{ and }\qquad \xi=\delta^{-\frac{1}{p}}\,u^{p-1}\,\frac{|\nabla w|^{p-2}\, \nabla w}{(w+\varepsilon)^{p-1}},
\end{equation}
where $\delta>0$ is a free parameter. Thus this time we get
\[
\begin{split}
\int_\Omega \left[\frac{(p-1)\,|\nabla w_\Omega|^p+(w_\Omega+\varepsilon)}{(w_\Omega+\varepsilon)^p} \right]\,|u|^p\,dx&\le \delta\,\int_\Omega |\nabla
u|^p\,dx+(p-1)\,\delta^{-\frac{1}{p-1}}\,\int_\Omega \frac{|u|^p\,|\nabla w_\Omega|^p}{(w_\Omega+\varepsilon)^{p}}\,dx.
\end{split}
\]
We can now pass to the limit on both sides. By using \eqref{integrability} and the Monotone Convergence Theorem we get
\[
\int_\Omega \left[\frac{(p-1)\,|\nabla w_\Omega|^p+w_\Omega}{w_\Omega^p} \right]\,|u|^p\,dx\le \delta\,\int_\Omega |\nabla
u|^p\,dx+(p-1)\,\delta^{-\frac{1}{p-1}}\,\int_\Omega |u|^p\,\left|\frac{\nabla w_\Omega}{w_\Omega}\right|^p\,dx.
\]
This gives the conclusion for $u$ smooth and positive. A density argument gives again the general result.
\end{proof}
\begin{remark}
Observe that one could optimize \eqref{HardyButtazzo} with respect to $\delta>0$. This leads to the following stronger form of the torsional Hardy inequality
\begin{equation}
\label{optimalHB}
\left(\frac{p-1}{p}\right)^{p}\,\frac{\displaystyle\left(\int_\Omega \left[\left|\frac{\nabla w_\Omega}{w_\Omega}\right|^p+\frac{1}{(p-1)\,w_\Omega^{p-1}} \right]\,|u|^p\,dx\right)^p}{\left(\displaystyle\int_\Omega \left|\frac{\nabla w_\Omega}{w_\Omega}\right|^p\,|u|^p\,dx\right)^{p-1}}\le \int_\Omega |\nabla
u|^p\,dx.
\end{equation}
We leave the details to the interested reader.
\end{remark}
\subsection{General case}
Finally, we consider the case of a general open set $\Omega\subset\mathbb{R}^N$. We will need the following version of the torsional Hardy inequality.
\begin{proposition}
\label{prop:HBgeneral}
Let $1<p<+\infty$ and let $\Omega\subset\mathbb{R}^N$ be an open set. Then for every $u\in W^{1,p}_0(\Omega)$ we have 
\begin{equation}
\label{HBgeneral}
\int_{\{w_\Omega<+\infty\}} \frac{|u|^p}{w_\Omega^{p-1}}\,dx\le \int_\Omega |\nabla u|^p\,dx.
\end{equation}
\end{proposition}
\begin{proof}
We just need to prove \eqref{HBgeneral} for functions in $C^\infty_0(\Omega)$. Let $u$ be a such a function, then we take as always $R_1>R_\Omega$ large enough so that the support of $u$ is contained in $\Omega\cap B_{R}$, for every $R\ge R_1$. We can then use Theorem \ref{thm:delta} with $\delta=1$ and obtain
\[
\int_{\Omega} \frac{|u|^p}{w_{\Omega,R}^{p-1}}\,dx\le \int_\Omega |\nabla u|^p\,dx,\qquad R\ge R_1.
\]
If we now take the limit as $R$ go to $\infty$ and use Fatou's Lemma once again, we get the desired conclusion by appealing to the definition of $w_\Omega$.
\end{proof}

\section{Proofs of the main results}
\label{sec:5}

\subsection{Proof of Theorem \ref{thm:main}}
For ease of notation, we set
\[
\gamma:=\frac{p-1}{p-q}\,q.
\]
We start by proving the first equivalence, i.e.
\[
\lambda_{p,q}(\Omega)>0 \quad \Longleftrightarrow \quad w_\Omega\in L^\gamma(\Omega).
\]
Let us assume that $\lambda_{p,q}(\Omega)>0$. We recall that $w_{\Omega,R}$ satisfies
\[
\int_{\Omega\cap B_R} \langle|\nabla w_{\Omega,R}|^{p-2}\, \nabla w_{\Omega,R},\nabla \phi\rangle\,dx=\int_{\Omega\cap B_R} \phi\,dx,
\]
for every $\phi\in W^{1,p}_0(\Omega\cap B_R)$. By Lemma \ref{lm:VF}, the function $\phi=w_{\Omega,R}^\beta$ is a legitimate test function for every $\beta\ge 1$, since $\Omega\cap B_R$ is an open bounded set. By using this, we get with simple manipulations 
\begin{equation}
\label{brividi}
\beta\,\left(\frac{p}{\beta+p-1}\right)^p\,\int_{\Omega\cap B_R} \left|\nabla w_{\Omega,R}^\frac{\beta+p-1}{p}\right|^p\,dx=\int_{\Omega\cap B_R} w_{\Omega,R}^\beta\,dx.
\end{equation}
We now observe that $(\beta+p-1)/p\ge 1$, thus  $w_{\Omega,R}^{(\beta+p-1)/p}\in W^{1,p}_0(\Omega\cap B_R)$, still thanks to Lemma \ref{lm:VF}. Moreover, the inclusion $\Omega\cap B_R\subset\Omega$ implies
\[
0<\lambda_{p,q}(\Omega)\le \lambda_{p,q}(\Omega\cap B_R).
\]
Then we can apply the relevant Poincar\'e inequality in the left-hand side of \eqref{brividi} and get
\[
\beta\,\left(\frac{p}{\beta+p-1}\right)^p\,\lambda_{p,q}(\Omega) \,\left(\int_{\Omega\cap B_R} w_{\Omega,R}^{\frac{\beta+p-1}{p}\,q}\,dx\right)^\frac{p}{q}\le \int_{\Omega\cap B_R} w_{\Omega,R}^{\beta}\,dx.
\]
This is valid for a generic $\beta\ge 1$. In order to obtain the desired estimate, we now choose
\[
\beta=\gamma=\frac{p-1}{p-q}\,q\qquad \mbox{ so that }\qquad \frac{\beta+p-1}{p}\,q=\beta,
\]
which is feasible, since $\gamma\ge 1$.
By using that $p/q>1$, after a simplification we get
\[
\lambda_{p,q}(\Omega)\, \left(\int_{\Omega\cap B_R} w_{\Omega,R}^\gamma\,dx\right)^\frac{p-q}{q}\le \frac{1}{q}\, \left(\frac{p-1}{p-q}\right)^{p-1}.
\]
We now take the limit ar $R$ goes to $+\infty$, then Fatou's Lemma gives that $w_\Omega\in L^\gamma(\Omega)$, together with the upper bound in \eqref{lambda1}.
\vskip.2cm\noindent
Let us now assume $w_\Omega\in L^\gamma(\Omega)$. We observe that the latter entails $|\{w_\Omega<+\infty\}|=|\Omega|$. Then for every $u\in C^\infty_0(\Omega)$ with unit $L^q$ norm, by combining H\"older inequality and \eqref{HBgeneral}, we get
\[
\begin{split}
\int_\Omega |u|^q\,dx=1&\le \left(\int_\Omega \frac{|u|^p}{w_\Omega^{p-1}}\,dx\right)^\frac{q}{p}\,\left(\int_\Omega w_\Omega^\gamma\,dx\right)^\frac{p-q}{p}\le \left(\int_\Omega |\nabla u|^p\,dx\right)^\frac{q}{p}\,\left(\int_\Omega w_\Omega^\gamma\,dx\right)^\frac{p-q}{p}.
\end{split}
\] 
By taking the infimum over admissible $u$, we get $\lambda_{p,q}(\Omega)>0$. The result comes with the lower bound in \eqref{lambda1}. 
\vskip.2cm\noindent
In order to complete the proof, we are now going to prove the equivalence
\[
\lambda_{p,q}(\Omega)>0 \quad \Longleftrightarrow \quad W^{1,p}_0(\Omega)\hookrightarrow L^q(\Omega) \mbox{ is compact}.
\]
The fact that the compactness of the embedding $W^{1,p}_0(\Omega)\hookrightarrow L^q(\Omega)$ implies $\lambda_{p,q}(\Omega)>0$ is straightforward. We thus focus on the converse implication.
Let us assume that $\lambda_{p,q}(\Omega)>0$. We remark that we already know that this implies (and is indeed equivalent to) $w_\Omega\in L^\gamma(\Omega)$. Let $\{u_n\}\subset W^{1,p}_0(\Omega)$ be a bounded sequence, i.e.
\begin{equation}
\label{gradientibound}
\|\nabla u_n\|_{L^p(\Omega)}\le L,\qquad \mbox{ for every }n\in\mathbb{N}.
\end{equation}
By definition of $\lambda_{p,q}(\Omega)$, we have that $\{u_n\}_{n\in\mathbb{N}}$ is bounded in $L^q(\Omega)$ as well. By uniform convexity of $W^{1,p}_0(\Omega)$, we have that $\{u_n\}_{n\in\mathbb{N}}$ converges weakly (up to a subsequence) in $W^{1,p}_0(\Omega)$ to a function $u\in W^{1,p}_0(\Omega)$. Moreover, we can suppose that the convergence is strong in $L^q_{loc}(\Omega)$. By Fatou's Lemma, this in turn implies that $u\in L^q(\Omega)$ as well. 
Let us consider the new sequence 
\[
U_n:=u_n-u\in W^{1,p}_0(\mathbb{R}^N)\cap L^q(\mathbb{R}^N),
\]
that we consider extended by $0$ outside $\Omega$. By  the local convergence in $L^q$, for every $\varepsilon>0$ and every $R>0$ there exists $n_{\varepsilon,R}\in\mathbb{N}$ such that
\begin{equation}
\label{dentro}
\int_{B_{R+1}} |U_n|^q\,dx< \varepsilon,\qquad \mbox{ for every }n\ge n_{\varepsilon,R}.
\end{equation}
In order to control the integral on $\mathbb{R}^N\setminus B_{R+1}$ uniformly, we use again the torsional Hardy inequality. For every $R>0$, we take a positive function $\eta_R\in C^{\infty}(\mathbb{R}^N\setminus B_R)$ such that
\[
\eta_R\equiv 1 \mbox{ in } \mathbb{R}^N\setminus B_{R+1},\qquad \eta_R\equiv 0 \mbox{ in } B_{R},\qquad 0\le \eta_R\le 1,\qquad |\nabla \eta_R|\le C,
\]
for some universal constant $C>0$. Each function $U_n\,\eta_R$ belongs to $W^{1,p}_0(\Omega)$, then by combining H\"older inequality and \eqref{HBgeneral} as before, we have
\begin{equation}
\label{boundare}
\begin{split}
\int_{\mathbb{R}^N\setminus B_{R+1}} |U_n|^q\,dx&\le  \left(\int_{\Omega} \frac{|U_n\,\eta_R|^p}{w_\Omega^{p-1}}\,dx\right)^\frac{q}{p}\,\left(\int_{\Omega\setminus B_R} w_\Omega^\gamma\,dx\right)^\frac{p-q}{p}\\
&\le \left(\int_\Omega |\nabla (U_n\,\eta_R)|^p\,dx\right)^\frac{q}{p}\,\left(\int_{\Omega\setminus B_R} w_\Omega^\gamma\,dx\right)^\frac{p-q}{p}.
\end{split}
\end{equation}
In the first inequality we used the properties of $\eta_R$, which imply in particular that $U_n\,\eta_R\equiv 0$ inside $B_R$.
We now observe that the first term in the right-hand side of \eqref{boundare} is bounded uniformly. Indeed, by \eqref{gradientibound} and the triangle inequality 
\[
\left(\int_\Omega |\nabla (U_n\,\eta_R)|^p\,dx\right)^\frac{1}{p}\le L+C\,\|U_n\|_{L^p(\Omega)},
\]
and the last term is again estimated in terms of $L$ and $\lambda_{p,q}(\Omega)>0$, thanks to the Gagliardo-Nirenberg inequalities of Proposition \ref{prop:GNS}, applied with $r=p$. 
\par
On the other hand, since $w_\Omega\in L^\gamma(\Omega)$, by the absolute continuity of the integral for every $\varepsilon>0$, there exists $R_\varepsilon>0$ such that
\[
\left(\int_{\Omega\setminus B_{R_\varepsilon}} w_\Omega^\gamma\,dx\right)^\frac{p-q}{p}\le \varepsilon.
\]
By spending these informations in \eqref{boundare}, we finally get 
\begin{equation}
\label{fuori}
\int_{\mathbb{R}^N\setminus B_{R_\varepsilon+1}} |U_n|^q\,dx< \widetilde C\,\varepsilon,\qquad \mbox{ for every }n\in\mathbb{N},
\end{equation}
for some $\widetilde C>0$ independent of $n$ and $\varepsilon$. By collecting \eqref{dentro} and \eqref{fuori}, we proved that for every $\varepsilon>0$ there exists $R_\varepsilon>0$ and $n_{\varepsilon}\in\mathbb{N}$ such that
\[
\int_{\mathbb{R}^N} |U_n|^q\,dx=\int_{B_{R_\varepsilon+1}} |U_n|^q\,dx+\int_{\mathbb{R}^N\setminus B_{R_\varepsilon+1}} |U_n|^q\,dx< (1+\widetilde C)\,\varepsilon,\qquad \mbox{ for every }n\ge n_\varepsilon.
\]
This finally shows that $U_n=u_n-u$ strongly converges to $0$ in $L^q(\Omega)$.

\subsection{Proof of Theorem \ref{thm:collection}} 
The fact that $w_\Omega\in L^\infty(\Omega)$ implies $\lambda_{p,p}(\Omega)>0$ follows as before by using the torsional Hardy inequality \eqref{HBgeneral}. Indeed, for every $u\in C^\infty_0(\Omega)$ with unit $L^p$ norm we have 
\[
\begin{split}
\int_\Omega |u|^p\,dx=1&\le\|w_\Omega\|_{L^\infty}^{p-1}\, \int_\Omega \frac{|u|^p}{w_\Omega^{p-1}}\,dx\,\le \|w_\Omega\|^{p-1}_{L^\infty}\,\int_\Omega |\nabla u|^p\,dx.
\end{split}
\]
This also shows the first inequality in \eqref{comprise}.
The converse implication is exactly the van den Berg--Bucur estimate of \cite[Theorem 9]{vabu}. 
\vskip.2cm\noindent
As for the characterization \eqref{compatto} of the compact embedding $W^{1,p}_0(\Omega)\hookrightarrow L^p(\Omega)$, we first observe that when this holds, then $\lambda_{p,p}(\Omega)>0$ and this in turn implies $w_\Omega\in L^\infty(\Omega)$.
The proof of the implication ``$\Longrightarrow$'' can now be proved exactly as in \cite[Theorem 6.1]{bubu} by Bucur and Buttazzo.
\par 
The implication ``$\Longleftarrow$'' can be proved by appealing again to the torsional Hardy inequality.
Indeed, the hypothesis on $w_\Omega$ implies that bounded sequences $\{u_n\}_{n\in\mathbb{N}}\subset W^{1,p}_0(\Omega)$ are bounded in $L^p(\Omega)$ as well, since $w_\Omega \in L^\infty(\Omega)$ and thus $\lambda_{p,p}(\Omega)>0$. Moreover, the bound on the $L^p$ norms of the gradients guarantees that translations converge to $0$ in $L^p(\Omega)$ uniformly in $n$, i.e.
\[
\lim_{|h|\to 0}\left(\sup_{n\in\mathbb{N}} \int_{\mathbb{R}^N} |u_n(x+h)-u_n(x)|^p\,dx\right)=0.
\] 
In order to exclude loss of mass at infinity for the sequence $\{|u_n|^p\}_{n\in\mathbb{N}}$, we observe that with an argument similar to that of \eqref{boundare}, by \eqref{HBgeneral} we have
\[
\int_{\mathbb{R}^N\setminus B_{R+1}} |u_n|^p\,dx\le \left(\int_\Omega |\nabla (u_n\,\eta_R)|^p\,dx\right)\,\|w_\Omega\|_{L^\infty(\Omega\setminus B_R)}^{p-1},
\]
where $\eta_R$ is as in the proof of Theorem \ref{thm:main}. Thus the loss of mass at infinity is excluded, by using the hypothesis on the decay at infinity of $w_\Omega$. This yields strong convergence in $L^p(\Omega)$ (up to a subsequence), thanks to the Riesz-Fr\'echet-Kolmogorov Theorem.

\begin{remark}
\label{rem:guidedonde}
Differently from the case $1\le q<p$, the fact that $\lambda_{p,p}(\Omega)>0$ {\it does not} entail in general that the embedding $W^{1,p}_0(\Omega)\hookrightarrow L^p(\Omega)$ is compact. A simple counterexample is given by any rectilinear wave-guide $\Omega=\omega\times \mathbb{R}\subset\mathbb{R}^N$, where $\omega\subset\mathbb{R}^{N-1}$ is a bounded open set. Indeed, it is well-known that $\lambda_{p,p}(\Omega)>0$ in this case, while one can find bounded sequences $\{u_n\}_{n\in\mathbb{N}}\subset W^{1,p}_0(\Omega)$ which do not admit strongly convergent subsequences. 
\end{remark}
\begin{example}
\label{exa:locale?}
For simplicity we focus on the case $p=2$, but the very same example works for every $1<p<+\infty$, with the necessary modifications. Let $\{r_i\}_{i\in\mathbb{N}}\subset \mathbb{R}$ be a sequence of strictly positive numbers, such that
\[
\sum_{i=0}^\infty r_i^N=+\infty.
\] 
We then define the sequence of points $\{x_i\}_{i\in\mathbb{N}}\subset\mathbb{R}^N$ by
\[
\left\{\begin{array}{rcl}
x_0&=&(0,\dots,0),\\
x_{i+1}&=&(r_i+r_{i+1},0,\dots,0)+x_i,
\end{array}
\right.
\]
and the set 
\begin{equation}
\label{palloni}
\Omega=\bigcup_{i=0}^\infty B_{r_i}(x_i),
\end{equation}
which by construction is a disjoint union of open balls, with $|\Omega|=+\infty$. On each ball $B_{r_i}(x_i)$ the torsion function is given by
\[
w_{B_{r_i}(x_i)}=\frac{(r_i^2-|x-x_i|^2)_+}{2\,N},
\]
thus we have the explicit expression for the torsion function of $\Omega$
\[
w_\Omega(x)=\sum_{i=0}^\infty w_{B_{r_i}(x_i)}=\sum_{i=0}^{\infty} \frac{(r_i^2-|x-x_i|^2)_+}{2\,N}.
\]
\begin{figure}[h]
\includegraphics[scale=.25]{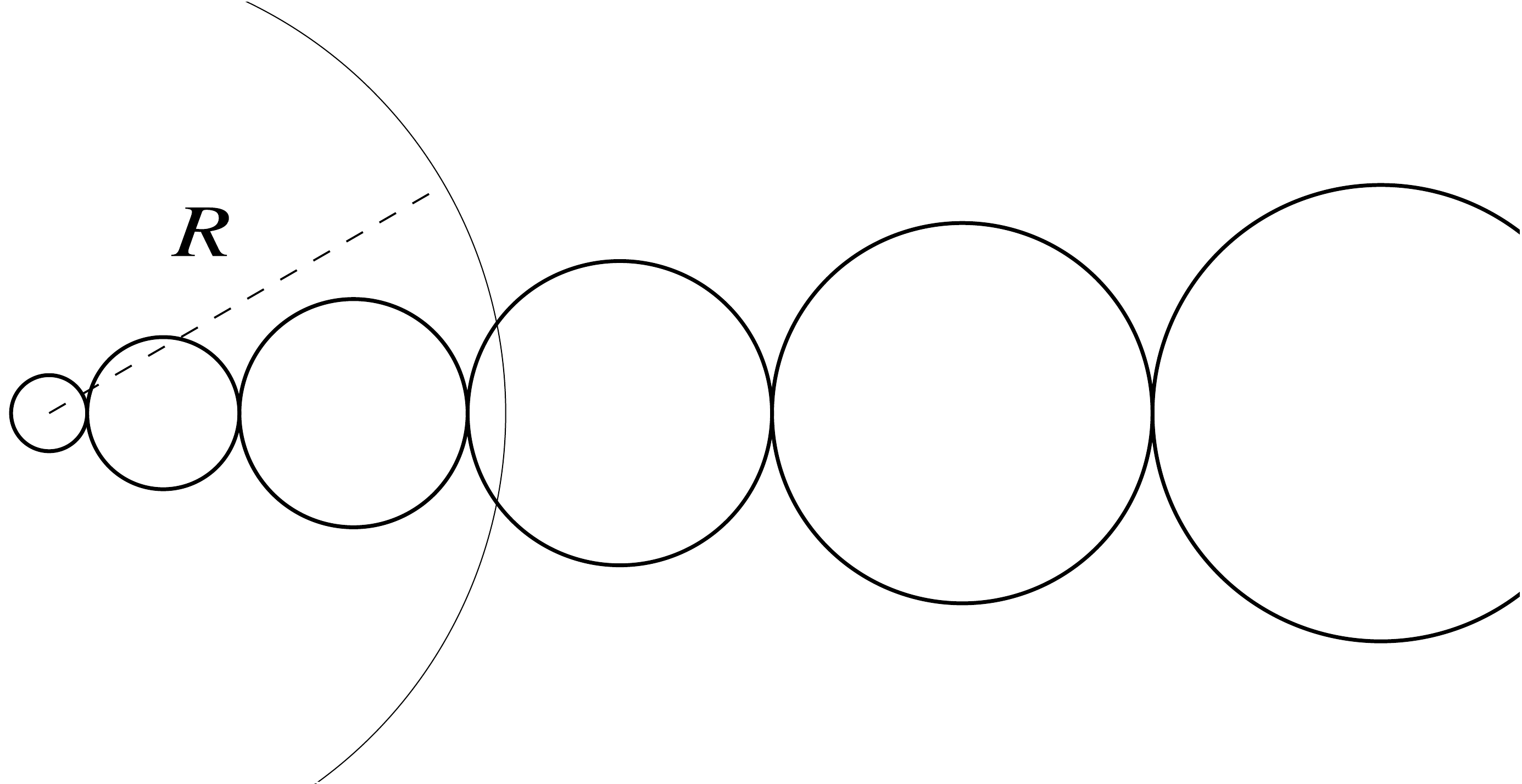}
\caption{The set of Example \ref{exa:locale?} when $r_i\nearrow +\infty$. The relevant torsion function is in $L^\infty_{loc}(\Omega)\setminus L^\infty(\Omega)$ and $\lambda_{p,p}(\Omega)=0$.}
\end{figure}
\par
We notice that $w_\Omega\in L^\infty_{loc}(\Omega)$ and we have
\[
w_\Omega\in L^\infty(\Omega)\qquad \Longleftrightarrow\qquad \limsup_{i\to\infty} r_i<+\infty.
\]
In this case $\lambda_{p,p}(\Omega)>0$ by Theorem \ref{thm:collection}.
For $1\le q<2$, by observing that
\[
\int_\Omega w_\Omega^\frac{q}{2-q}\,dx\simeq \sum_{i=1}^\infty \int_{B_{r_i}} (r_i^2-|x|^2)^\frac{q}{2-q}\,dx\simeq \sum_{i=1}^\infty r_i^{\frac{2\,q}{2-q}+N},
\] 
we have
\begin{equation}\label{sommabilitapalle}
w_\Omega\in L^\frac{q}{2-q}(\Omega)\qquad \Longleftrightarrow \qquad \sum_{i=1}^\infty r_i^{\frac{2\,q}{2-q}+N}<+\infty.
\end{equation}
In this case $\lambda_{p,q}(\Omega)>0$ and $W^{1,p}_0(\Omega)\hookrightarrow L^q(\Omega)$ is compact by Theorem \ref{thm:main}.
\end{example}

\begin{remark}
\label{rem:trenoedisperazione}
By exploiting Example \ref{exa:locale?}, it is not difficult to show that for every $0<s<1$ there exists an open set $\Omega\subset\mathbb{R}^N$ such that $w_\Omega\in L^1(\Omega)\setminus L^s(\Omega)$ (this in particular means that the embedding $W^{1,p}_0(\Omega)\hookrightarrow L^1(\Omega))$ is compact, by Theorem \ref{thm:main}). Indeed, with the notations of the previous example in force, by taking $\Omega$ as in \eqref{palloni} we only have to show that there exists a sequence $\{r_i\}_{i\in\mathbb N}$ such that 
\[
\sum_{i=1}^{\infty} r_i^{2+N}<+\infty\qquad\mbox{ and }\qquad  \sum_{i=1}^{\infty} r_i^{2\,s+N}=\infty.
\]
Thanks to \eqref{sommabilitapalle}, this would entail that $w_\Omega\in L^1(\Omega)$, while $w_\Omega\not\in L^s(\Omega)$ (observe that $s=q/(2-q)$ implies $2\,q/(2-q)=2\,s$).
As a straightforward computation shows, an example of such a sequence is offered by the choice $r_i=i^{-1/(2\,s+N)}$. 
\end{remark}

\section{Sharpness of the torsional Hardy inequality}
\label{sec:6}

Since the essential ingredient of the lower bounds in \eqref{lambda1} and \eqref{comprise} is the torsional Hardy inequality \eqref{HardyButtazzo}, it is natural to address the question of its sharpness. Though sharpness of \eqref{HardyButtazzo} is not a warranty of optimality of the estimates \eqref{lambda1} and \eqref{comprise}, we believe this question to be of independent interest.
As we will see, the following value of the parameter $\delta>0$ in \eqref{HardyButtazzo}
\[
\delta=\left(\frac{p}{p-1}\right)^{p-1},
\]
will play a crucial role. We warn the reader that for simplicity in this section we will make the stronger assumption $|\Omega|<+\infty$.
\vskip.2cm\noindent
We start with a standard consequence of the Harnack inequality.
\begin{lemma}
\label{u>0}
Let $1<p<+\infty$ and let $\Omega\subset\mathbb{R}^N$ be an open connected set with finite measure. Let $\delta\ge 1$ and suppose that $u\in W^{1,p}_0(\Omega)$ is a nontrivial function attaining the equality in \eqref{suboptimalHB}. Then $u$ does not change sign and for every compact set $K\Subset\Omega$ there exists a constant $C_{K,u}>0$ such that 
\begin{equation}
\label{dalbasso}
|u|\ge \frac{1}{C_{K,u}},\qquad \mbox{ on }K.
\end{equation}
\end{lemma}
\begin{proof}
If $u\not\equiv 0$ is an optimal function, then $u$ minimizes in $W^{1,p}_0(\Omega)$ the functional
\[
\mathfrak{F}(\varphi)=\int_{\Omega}|\nabla \varphi|^p\,dx-\int_\Omega g\,|\varphi|^p\,dx,
\]
where
\[
g(x)=\frac{p-1}{\delta}\,\left[\left(1-\delta^{-\frac{1}{p-1}}\right)\,\left|\frac{\nabla w_\Omega}{w_\Omega}\right|^p+\frac{1}{(p-1)\,w_\Omega^{p-1}} \right].
\]
Since $\delta\ge1$, we have $g>0$. Then $v=|u|\in W^{1,p}_0(\Omega)$ is still a minimizer and it is of course positive. The relevant Euler-Lagrange equation is given by
\[
\left\{
\begin{array}{rcll}
-\Delta_p v&=&g\,v^{p-1},& \text{ in } \Omega,\\
v&=&0,& \mbox{ on } \partial \Omega.
\end{array}\right.
\]
In particular, $v$ is a positive local weak solution of 
\[
-\Delta_p v=g\,v^{p-1},
\]
and observe that $g\in L^\infty_{loc}(\Omega)$. Then $v$ satisfies Harnack inequality and thus $v\in C^{0,\alpha}_{loc}(\Omega)$ for some $0<\alpha<1$ (see for instance \cite[Theorem 1.1]{Tr}). Moreover, still by Harnack inequality $v$ verifies \eqref{dalbasso}. By recalling that $v=|u|$ and using the properties above for $v$, we get the desired result.
\end{proof}
\begin{proposition}[Existence of extremals]
Let $1<p<+\infty$ and let $\Omega\subset\mathbb{R}^N$ be an open set with finite measure. Assume that
\begin{equation}
\label{YES}
0<\delta<\left(\frac{p}{p-1}\right)^{p-1},
\end{equation}
then functions of the type
\begin{equation}
\label{estremali}
u=c\, w_\Omega^{\delta^{-\frac{1}{p-1}}},\qquad c\in\mathbb{R},
\end{equation}
give equality in \eqref{HardyButtazzo}.
\end{proposition}
\begin{proof}
We first observe that hypothesis \eqref{YES} implies that
\[
\delta^{-\frac{1}{p-1}}>\frac{p-1}{p},
\]
so that by Lemma \ref{lm:VF}, functions of the type \eqref{estremali} are in $W^{1,p}_0(\Omega)$. Then the proof is by direct verification. 
Indeed, let us take for simplicity $c=1$, then we get
\begin{equation}
\label{droite}
\int_\Omega |\nabla u|^p=\delta^{-\frac{p}{p-1}}\, \int_\Omega |\nabla w_\Omega|^p\, w_\Omega^{p\,\delta^{-\frac{1}{p-1}}-p}\,dx,
\end{equation}
and
\begin{equation}
\label{gauche}
\begin{split}
\frac{p-1}{\delta}\,\int_\Omega &\left[\left(1-\delta^{-\frac{1}{p-1}}\right)\,\left|\frac{\nabla w_\Omega}{w_\Omega}\right|^p+\frac{1}{(p-1)\,w_\Omega^{p-1}} \right]\,|u|^p\,dx\\
&=\frac{p-1}{\delta}\,\left(1-\delta^{-\frac{1}{p-1}}\right)\int_\Omega \,|\nabla w_\Omega|^p\, w_\Omega^{p\,\delta^{-\frac{1}{p-1}}-p}\,dx+\frac{1}{\delta}\,\int_\Omega w_\Omega^{p\,\delta^{-\frac{1}{p-1}}-p+1}\,dx.
\end{split}
\end{equation}
We now have to distinguish two cases:
\begin{equation}
\label{dupalle}
0<\delta<\left(\frac{p^2}{p^2-1}\right)^{p-1}\qquad \mbox{ or }\qquad \left(\frac{p^2}{p^2-1}\right)^{p-1}\le\delta<\left(\frac{p}{p-1}\right)^{p-1}.
\end{equation}
In the first case, we can insert in \eqref{weakeqw} the test function\footnote{This is a legitimate test function by Lemma \ref{lm:VF}, since 
\[
p\,\delta^{-\frac{1}{p-1}}-p+1>\frac{p-1}{p}\qquad\Longleftrightarrow\qquad 0<\delta<\left(\frac{p^2}{p^2-1}\right)^{p-1},
\]
and the latter is true by hypothesis.}
\begin{equation}
\label{test!}
\phi=w_\Omega^{p\,\delta^{-\frac{1}{p-1}}-p+1},
\end{equation}
then we get
\begin{equation}
\label{salvi}
\frac{1}{\delta}\,\int_\Omega w_\Omega^{p\,\delta^{-\frac{1}{p-1}}-p+1}\,dx=\left[-\frac{p-1}{\delta}\,\left(1-\delta^{-\frac{1}{p-1}}\right)+\delta^{-\frac{p}{p-1}}\right]\,\int_\Omega |\nabla w_\Omega|^p\,w_\Omega^{p\,\delta^{-\frac{1}{p-1}}-p}\,dx.
\end{equation}
By using this in \eqref{gauche} and comparing with \eqref{droite}, we get the conclusion.
\vskip.2cm\noindent
If on the contrary the second condition in \eqref{dupalle} is verified, some care is needed. Indeed, now the choice \eqref{test!} is not feasible for the equation \eqref{weakeqw}. We thus need to replace it by
\[
\phi_n=(w_\Omega+\varepsilon_n)^{p\,\delta^{-\frac{1}{p-1}}-p+1}-\varepsilon_n^{p\,\delta^{-\frac{1}{p-1}}-p+1},
\] 
where $\{\varepsilon_n\}_{n\in\mathbb{N}}\subset (0,+\infty)$ is an infinitesimal strictly decreasing sequence.
Then from \eqref{weakeqw} we get
\[
\begin{split}
\frac{1}{\delta}\,\int_\Omega &\Big[(w_\Omega+\varepsilon_n)^{p\,\delta^{-\frac{1}{p-1}}-p+1}-\varepsilon_n^{p\,\delta^{-\frac{1}{p-1}}-p+1}\Big]\,dx\\
&=\left[-\frac{p-1}{\delta}\,\left(1-\delta^{-\frac{1}{p-1}}\right)+\delta^{-\frac{p}{p-1}}\right]\,\int_\Omega |\nabla w_\Omega|^p\,(w_\Omega+\varepsilon_n)^{p\,\delta^{-\frac{1}{p-1}}-p}\,dx.
\end{split}
\]
If we now use the Dominated Convergence Theorem on both sides, we obtain as before \eqref{salvi} and thus we get again the desired conclusion.
\end{proof}
\begin{proposition}[Lack of extremals]
\label{prop:sharpNO}
Let $1<p<+\infty$ and let $\Omega\subset\mathbb{R}^N$ be an open connected set with finite measure. If
\begin{equation}
\label{NO}
\delta\ge\left(\frac{p}{p-1}\right)^{p-1},
\end{equation}
equality in \eqref{HardyButtazzo} is not attained in $W^{1,p}_0(\Omega)\setminus\{0\}$.
\end{proposition}
\begin{proof}
We first notice that by using the quantitative version of Young inequality (see Propositions \ref{prop:giovanesopra2} and \ref{prop:giovanesotto2} below) in place of \eqref{giovane!}, we can show the following stronger version of \eqref{HardyButtazzo} 
\[
\begin{split}
\frac{p-1}{\delta}\,\int_\Omega \left[\left(1-\delta^{-\frac{1}{p-1}}\right)\,\left|\frac{\nabla w_\Omega}{w_\Omega}\right|^p+\frac{1}{(p-1)\,w_\Omega^{p-1}} \right]\,|u|^p+\frac{C}{\delta}\, \mathcal{R}_{p,\Omega}(u)&\le \int_\Omega |\nabla
u|^p\,dx.\\
\end{split}
\]
Here $C=C(p)>0$ is a constant and the remainder term $\mathcal{R}_{p,\Omega}(u)$ is given by
\[
\begin{split}
\mathcal{R}_{p,\Omega}(u)=\int_\Omega &\left|\delta^\frac{1}{p}\,\nabla u-\delta^{-\frac{1}{p\,(p-1)}}\,u\,\frac{\nabla w_\Omega}{w_\Omega}\right|^p\,dx,\qquad \mbox{ if } p\ge 2,
\end{split}
\]
or
\[
\begin{split}
\mathcal{R}_{p,\Omega}(u)&=\int_\Omega\left[\delta^\frac{2}{p}\,|\nabla u|^2+\delta^{-\frac{2}{p\,(p-1)}}\,u^2\,\left|\frac{\nabla w_\Omega}{w_\Omega}\right|^2\right]^\frac{p-2}{2}\\ &\times\left|\delta^\frac{1}{p}\,\nabla u-\delta^{-\frac{1}{p-1}}\,u\,\frac{\nabla w_\Omega}{w_\Omega}\right|^2\,dx,\qquad \mbox{ if } 1<p<2.
\end{split}
\]
Thus if $u\in W^{1,p}_0(\Omega)\setminus\{0\}$ is such that equality holds in \eqref{HardyButtazzo}, then necessarily $\mathcal{R}_{p,\Omega}(u)=0$. This yields
\[
\frac{\nabla u}{u}=\delta^{-\frac{1}{p-1}}\, \frac{\nabla w_\Omega}{w_\Omega},
\]
and observe that by Lemma \ref{u>0} we can suppose that $u$ is positive (by assumption we have $\delta>1$). Then we arrive at
\[
\log u=\delta^{-\frac{1}{p-1}}\,\log w_\Omega+c,\qquad \mbox{ a.\,e. in }  \Omega.
\]
Notice that it is possible to divide by $u$ thanks to Lemma \ref{u>0}. From the previous identity we obtain 
\begin{equation}
\label{proportional}
u=c'\, w_\Omega^{\delta^{-\frac{1}{p-1}}},\qquad 
\end{equation}
almost everywhere in $\Omega$ for some constant $c'\not=0$. Observe that the hypothesis \eqref{NO} on $\delta$ implies that $\delta^{-1/(p-1)}\le (p-1)/p$. Thus 
thanks to Lemma \ref{lm:VF} we get a contradiction with the fact that $u\in W^{1,p}_0(\Omega)$.
\end{proof}
\begin{remark}
We recall that the weaker information ``{\it $u>0$ almost everywhere}'' could not be sufficient to conclude \eqref{proportional}. Indeed, one can find functions $u$ and $v$ such that
\[
\nabla u=u\,\frac{\nabla v}{v},\quad \mbox{ for a.\,e. }x\in\Omega\qquad \mbox{ and }\qquad u,v>0\quad \mbox{ for a.\,e. }x\in\Omega,
\] 
but $u$ and $v$ are not proportional. A nice example of this type is in \cite[page 84]{LP}.
\end{remark}
Finally, let us give a closer look at the borderline case
\[
\delta=\left(\frac{p}{p-1}\right)^{p-1}.
\] 
In this case \eqref{HardyButtazzo} reduces to \eqref{suboptimalHB}. We already know that equality can not be attained. Nevertheless, the inequality is sharp.
\begin{proposition}[Borderline case]
\label{prop:sharpweak}
Let $1<p<+\infty$ and let $\Omega\subset\mathbb{R}^N$ be an open set with finite measure. There exists a sequence $\{u_n\}_{n\in\mathbb{N}}\subset W^{1,p}_0(\Omega)\setminus\{0\}$ such that
\begin{equation}
\label{successione}
\lim_{n\to\infty}\frac{\displaystyle\int_\Omega |\nabla
u_n|^p\,dx}{\displaystyle \int_\Omega \left[\left|\frac{\nabla w_\Omega}{w_\Omega}\right|^p+\frac{p}{(p-1)\,w_\Omega^{p-1}} \right]\,|u_n|^p\,dx}=
\left(\frac{p-1}{p}\right)^{p}.
\end{equation}
\end{proposition}
\begin{proof}
Let us consider the sequence of functions in $W^{1,p}_0(\Omega)$ given by 
\[
u_n=w_\Omega^{\frac{p-1}{p}+\frac{1}{n}},\qquad n\in\mathbb{N}.
\]
Observe that these functions belong to $W^{1,p}_0(\Omega)$ thanks to\footnote{The assumption $|\Omega|<+\infty$ guarantees that $w_\Omega\in L^q(\Omega)$, for every $0<q<1$.} Lemma \ref{lm:VF}.
We have
\[
u_n^p=w_\Omega^{p-1+\frac{p}{n}}\qquad \mbox{ and }\qquad |\nabla u_n|^p=\left(\frac{p-1}{p}+\frac{1}{n}\right)^p\,w_\Omega^{-1+\frac{p}{n}}|\nabla w_\Omega|^p.
\]
So we get
\[
\begin{split}
\left(\frac{p-1}{p}\right)^p 
 &\le \left(\frac{p-1}{p}+\frac{1}{n}\right)^p\left(\int_\Omega \frac{|\nabla
w_\Omega|^p}{w_\Omega^{1-\frac{p}{n}}}\,dx\right)\left(\int_\Omega \left[\left|\frac{\nabla w_\Omega}{w_\Omega}\right|^p+\frac{p}{(p-1)\,w_\Omega^{p-1}} \right]\,w_\Omega^{p-1+\frac{p}{n}}\,dx\right)^{-1}\\
&=\left(\frac{p-1}{p}+\frac{1}{n}\right)^p\left(\int_\Omega \frac{|\nabla
w_\Omega|^p}{w_\Omega^{1-\frac{p}{n}}}\,dx\right)\left(\int_\Omega \frac{|\nabla
w_\Omega|^p}{w_\Omega^{1-\frac{p}{n}}}\,dx+\int_\Omega \frac{p}{p-1}\,w_\Omega^\frac{p}{n}\,dx\right)^{-1}\\
&\le \left(\frac{p-1}{p}+\frac{1}{n}\right)^p.
 \end{split}
\]
By taking the limit as $n$ goes to $\infty$, we conclude \eqref{successione}.
\end{proof}
\begin{remark}
Actually, we are not able to decide whether sharpness holds {\it in the whole range}
\[
\delta\ge\left(\frac{p}{p-1}\right)^{p-1}.
\]
We leave this as an interesting open question.
\end{remark}

\appendix

\section{Convexity inequalities}
\label{sec:A}

In order to make the paper self-contained, we recall some results about uniform convexity of power functions. The proofs are well-known, thus we mainly omit them.
\subsection{Case $p\ge 2$}

\begin{lemma}
\label{lm:superquadrato}
Let $p\ge 2$. For every $z,w\in\mathbb{R}^N$ we have
\begin{equation}
\label{lambdaconvex}
\frac{1}{2}\,|z|^p+\frac{1}{2}\, |w|^p\ge \left|\frac{z+w}{2}\right|^p+C\,\left(|z|^2+|w|^2\right)^\frac{p-2}{2}\,|z-w|^2,
\end{equation}
for some constant $C=C(p)>0$.
\end{lemma}
\begin{proposition}[Young inequality with remainder term]
\label{prop:giovanesopra2}
Let $p\ge 2$. For every $z,\xi\in\mathbb{R}^N$ we have
\begin{equation}
\label{young}
\langle \xi,z\rangle\le \frac{1}{p}\, |z|^p+\frac{1}{p'}\,|\xi|^{p'}-\frac{2}{p}\,C\, \left(|z|^2+|\xi|^\frac{2}{p-1}\right)^\frac{p-2}{2}\,\left|z-|\xi|^{p'-2}\,\xi\right|^2,
\end{equation}
where $C=C(p)>0$ is the constant appearing in \eqref{lambdaconvex}. In particular, we also have
\begin{equation}
\label{young2}
\langle \xi,z\rangle\le \frac{1}{p}\, |z|^p+\frac{1}{p'}\,|\xi|^{p'}-C\, \left|z-|\xi|^{p'-2}\,\xi\right|^p,
\end{equation}
possibly with a different $C=C(p)>0$.
\end{proposition}
\begin{proof}
By using the ``above tangent property'' of a convex function in \eqref{lambdaconvex}, we get
\[
\langle |w|^{p-2}\,w,z\rangle\le \frac{1}{p}\,|z|^p+\left(1-\frac{1}{p}\right)\,|w|^{p}-\frac{2}{p}\,C\, (|z|^2+|w|^2)^\frac{p-2}{2}\, |z-w|^2.
\]
If we now make the choice $w=|\xi|^{p'-2}\,\xi$ in the previous inequality, we get the desired conclusion \eqref{young}.
\vskip.2cm
\noindent
In order to prove \eqref{young2}, it is sufficient to observe that by using the concavity of $t\mapsto \sqrt{t}$ and monotonicity of $t\mapsto t^{p-2}$, we get
\[
\left(|z|^2+|\xi|^\frac{2}{p-1}\right)^\frac{p-2}{2}\ge 2^\frac{2-p}{2}\,\left(|z|+|\xi|^\frac{1}{p-1}\right)^{p-2}.
\]
On the other hand, by triangle inequality
\[
\begin{split}
\left|z-|\xi|^{p'-2}\,\xi\right|^p
&\le \left(|z|+|\xi|^\frac{1}{p-1}\right)^{p-2}\, \left|z-|\xi|^{p'-2}\,\xi\right|^2.
\end{split}
\]
By using these two inequalities in \eqref{young}, we get \eqref{young2}.
\end{proof}

\subsection{Case $1<p<2$}

\begin{lemma}
Let $1<p<2$. For every $z,w\in\mathbb{R}^N$ such that $|z|^2+|w|^2\not=0$ we have
\begin{equation}
\label{lambdaconvex_sub}
\frac{1}{2}\,|z|^p+\frac{1}{2}\, |w|^p\ge \left|\frac{z+w}{2}\right|^p+C\,\left(|z|^2+|w|^2\right)^\frac{p-2}{2}\,|z-w|^2,
\end{equation}
for some constant $C=C(p)>0$.
\end{lemma}
\begin{proposition}[Young inequality with remainder term]
\label{prop:giovanesotto2}
Let $1<p<2$. For every $z,\xi\in\mathbb{R}^N$ such that $|z|^2+|\xi|^2\not=0$ we have
\begin{equation}
\label{young_sub}
\langle \xi,z\rangle\le \frac{1}{p}\, |z|^p+\frac{1}{p'}\,|\xi|^{p'}-\frac{2}{p}\,C\, \left(|z|^2+|\xi|^\frac{2}{p-1}\right)^\frac{p-2}{2}\,\left|z-|\xi|^{p'-2}\,\xi\right|^2,
\end{equation}
where $C=C(p)>0$ is the constant appearing in \eqref{lambdaconvex_sub}. 
\end{proposition}
\begin{proof}
The proof of \eqref{young_sub} is exactly the same as that of \eqref{young} and we omit it.
\end{proof}

\end{document}